\documentclass[reqno,11pt]{amsart}
\usepackage{amsthm}
\usepackage{amsmath,amsfonts,amssymb,bm}
\usepackage{xr}
\usepackage{mathrsfs}
\usepackage{mathtools}
\usepackage{hyperref}
\usepackage{mathabx}
\usepackage[initials,nobysame]{amsrefs}
\usepackage[top=2.0cm,bottom=2.0cm,left=3cm,right=3cm]{geometry}
\usepackage{amsthm}
\usepackage[T1]{fontenc}
\usepackage{tikz}
\usepackage{enumitem}
\usepackage{float}
\usetikzlibrary{arrows, bending}

\newtheorem{theorem}{Theorem}[section]
\newtheorem{lemma}[theorem]{Lemma}

\newtheorem{remark}[theorem]{Remark}
\numberwithin{figure}{section}
\numberwithin{equation}{section}

\mathtoolsset{showonlyrefs}
\makeatletter

\newcommand{\Rmnum}[1]{\expandafter\@slowromancap\romannumeral #1@}
\makeatother

\AtBeginDocument{
   \def\MR#1{}
}

\newcommand\Lpre{ \mathbb{P}_{\scriptscriptstyle\mathcal{L}}}
\newcommand\Loth{ \mathbb{P}^{\perp}_{\scriptscriptstyle\mathcal{L}}}

\newcommand{\mdot}{ \!\cdot\!}
\def\u{ \mathrm{u} }
\def\d{ \mathrm{d} }

\title{Relative entropy and slightly compressible Navier-Stokes dynamics of the Boltzmann equation}
\author[Yuhan Chen]{Yuhan Chen}
\address[Yuhan Chen]
{\newline School of Mathematics and Statistics, Wuhan University, Wuhan 430072, P. R. China}
\email{yhchen\_math@whu.edu.cn}

\author[Ning Jiang]{Ning Jiang}
\address[Ning Jiang]{\newline School of Mathematics and Statistics, Wuhan University, Wuhan 430072, P. R. China}
\email{njiang@whu.edu.cn}
\begin{document}
\begin{abstract}
    This paper shows that, in the formal level, the convergence of solutions of Boltzmann equation to solutions of the  compressible Navier-Stokes system with small Mach number over the three-dimensional periodic domain $\mathbb{T}^3$,
    using the relative entropy method originated from Bardos, Golse, Levermore [{\em Comm. Pure Appl. Math.} {\bf 46} (1993) 667--753] and Yau [{\em Lett. Math. Phys.} {\bf 22} (1991) 63--80]. We discuss the evolution of the entropy which is relative to the local Maxwellian governed by the solution of slightly compressible Navier-Stokes system. This characterizes the convergence rate from Boltzmann equation to the incompressible Navier-Stokes system. \\


 \end{abstract}

\maketitle
\thispagestyle{empty}
\pagestyle{plain}
    \section{Introduction}
In this paper, we study, in the formal level, the asymptotic behavior of the following rescaled Boltzmann equation to the compressible Navier-Stokes system with small Mach number. To avoid the difficulty might be brought from the boundary, we consider the periodic domain $\mathbb{T}^3$.
\begin{equation}\label{longtime_Boltzmann} \tag{BE}
    \left\{\begin{aligned}
        &\tau_{\varepsilon} \partial_t f_{\varepsilon} + v\mdot \nabla_x f_{\varepsilon} =\frac{1}{\varepsilon} C(f_{\varepsilon},\ f_{\varepsilon}), \quad x\in \mathbb{T}^3, \quad v \in \mathbb{R}^3\\
        &f_{\varepsilon}(0,x) = f_{\varepsilon}^{\mathrm{in}} \geqslant 0,
    \end{aligned}\right.
\end{equation}
 where $f_{\varepsilon} = f_{\varepsilon}(t,x,v)$ denotes the number density of gas molecules at position $ x\in \mathbb{T}^3$ and time $t \geqslant 0$. Moreover, $\varepsilon$ denotes the Knudsen number which is the ratio of the mean free path and the macroscopic length. $\tau_{\varepsilon}$ denotes the Mach number which is the ratio of the bulk velocity to the sound speed.  The case $\tau_{\varepsilon} =1 $ is corresponding to the short time scale (or Euler time scale), and in this paper we shall focus on the case $\tau_{\varepsilon} = \varepsilon $ which is corresponding to the longer time scale (or Navier-Stokes scale).
The smaller the Knudsen number is, the behavior of the Boltzmann equation is closer to the fluid equations. The limiting process when $\varepsilon$ goes to zero is called hydrodynamic limit.

The operator $ C(f,g)$ is the collision operator which describe the binary elastic collision between particles and is defined as follows
\begin{equation}
    \begin{aligned}
        C( f , g)  = \frac{1}{2}\iint_{\mathbb{T}^3_x \times \mathbb{S}^2} B(v-v_1 , \sigma) \left( g^\prime_1  f^\prime + f^\prime_1  g^\prime -g_1 f -f_1 g   \right) d\sigma dv_1.
    \end{aligned}
\end{equation}
Here and in the sequel we use the notation
\begin{equation}
    f^{\prime}_{1} = f(t,x,v^{\prime}_1), \ f^{\prime} = f(t,x,v^{\prime}),\ f_{1} = f(t,x,v_1),
\end{equation}
and for $\sigma \in \mathbb{S}^2 $,
\begin{equation}
    v^{\prime} = \frac{v+ v_1}{2}+ \frac{|v-v_1|}{2} \sigma,  \ v^{\prime}_1 = \frac{v+ v_1}{2} - \frac{|v-v_1|}{2} \sigma.
\end{equation}
The function $B(v-v_1,\sigma)$, called cross-section, is assumed to depend only on $v-v_1$ and $\frac{v-v_1}{|v-v_1|} \cdot \sigma$.

It is well-known that the equilibrium of the Boltzmann collision operator $C$, i.e. the number density $\mathcal{M}$ such that $C(\mathcal{M}, \mathcal{M})=0$, have a specific form, which is called the Maxwellian distributions:
$$\mathcal{M}(\rho, u ,\theta) :=\frac{\rho}{(2\pi\theta)^{\frac{3}{2}}}\exp(-\frac{|v-\u|^2}{2\theta})\,,$$
where $(\rho,\u,\theta)$ denotes the fluid variables: density, bulk velocity and temperature. When $(\rho,u,\theta)$ are functions of $(x,t)$, $\mathcal{M}(\rho, \u ,\theta)$ is called the local Maxwellians, while when $(\rho,\u,\theta)$ are constant states, $\mathcal{M}(\rho, \u ,\theta)$ is called the global (or absolute) Maxwellians.

In different scales for time and fluctuations relative to Maxwellians, the fluid dynamics of Boltzmann equations are represented by different equations, such as compressible/incompressible Navier-Stokes or Euler equations. As formally derived in \cite{BGL1}, when $\tau_{\varepsilon} = \varepsilon $, this is the case where Mach number and Knudsen number both equals to $\varepsilon$ (then the Reynold number $\mathrm{Re}=O(1)$ from the von Kármán relation $\mathrm{Re}=\frac{\mathrm{Ma}}{\mathrm{Kn}}$ )
 and the incompressible Navier-Stokes limit can be expected. When considering small Mach number, it is natural to consider distributions  as perturbations about a given absolute Maxwellian such as $M = \mathcal{M}(1,0,1)$. More specifically, the solutions to the equation \eqref{longtime_Boltzmann} are sought in the form $f_{\varepsilon} = M + \varepsilon M g_{\varepsilon}$.
Then the equation \eqref{longtime_Boltzmann} can be rewritten as
\begin{equation}\label{perturbed_Boltzmann}
    \left\{ \begin{matrix}
        \varepsilon^2 \partial_t g_{\varepsilon} + \varepsilon v\mdot \nabla_x g_{\varepsilon} + \mathcal{L}g_{\varepsilon} = \varepsilon Q(g_{\varepsilon},g_{\varepsilon}),\\
        g_{\varepsilon}(0,x,v) = g^{{in}}_{\varepsilon},
    \end{matrix}\right.
\end{equation}
with the linearized Boltzmann operator
\begin{equation}
    \mathcal{L}(g) := - \frac{1}{M}(C(M,Mg) + C(Mg,M)),
\end{equation}
and
\begin{equation}
    \begin{aligned}
        &Q(f,g) := \frac{1}{M}C(Mf,Mg).\\
    \end{aligned}
\end{equation}
We shall use the following notations in the sequel:
    \begin{equation}
        \begin{aligned}
            &\langle f,g \rangle = \int_{\mathbb{R}^3_v} f g M\,\d v, \ \ \
            \langle f \rangle = \int_{\mathbb{T}^3} f M \,\d v,\\
            &\langle\! \langle  f  \rangle\! \rangle := \int_{\mathbb{R}^3_v} \int_{\mathbb{R}^3_v}  \int_{\mathbb{S}^2}f M M _1 b \,\d v \d v_1 \d\sigma,
        \end{aligned}
    \end{equation}
    and the standard inner product $({a} \!\  , {b}) = \sum_{i=1}^5 a^ib^i$ in $\mathbb{R}^5$.

    It is well known that the linearized Boltzmann operator $\mathcal{L}$ is self-adjoint on $L^2(Mdv)$ with inner product $\langle f,g \rangle $.
    Moreover, the null space $\mathcal{N}$ of $\mathcal{L}$ is spanned by the set of collision invariants:
    \begin{equation}
        \begin{aligned}
            \mathcal{N} = \text{Span} \{ 1 , v, |v|^2\}.
        \end{aligned}
    \end{equation}
    We also let $\mathcal{N}^{\perp}$ to denote the orthogonal space of $\mathcal{N}$ with respect to the inner product $\langle \cdot, \cdot \rangle$ defined above.
    Furthermore, we use the notation $\Lpre g$ and $\Loth g$ to denote the projection of $g$ to $\mathcal{N}$ and $\mathcal{N}^{\perp}$ respectively.

    Formally, $g_{\varepsilon}$ should converge to some $g$ which lies in the kernel of $\mathcal{L}$. Specifically, $g$ should be of the form: $g = \rho + v \cdot \u  +( \frac{1}{2} |v|^2 -\frac{3}{2})\theta$. Moreover,
    the velocity $u$ is divergence free and the density and the temperature fluctuations $\rho$ and $\theta$ satisfy the Boussinesq relation
    \begin{equation}
        \nabla_x \cdot \u =0; \ \ \rho + \theta = 0.
    \end{equation}
    Furthermore, $(\rho,\u, \theta)$ should satisfy the incompressible Navier-Stokes-Fourier system
    \begin{equation}\label{INSF} \tag{INSF}
        \left\{
            \begin{aligned}
                &\partial_t \u + \u \mdot \nabla_x \u + \nabla_x p = \mu \Delta \u, \\
                &\partial_t \vartheta + \u \mdot \nabla_x \vartheta = \kappa \Delta \vartheta,\\
                &\nabla_x \cdot \u =0,
            \end{aligned}
        \right.
    \end{equation}
    where $\vartheta = \frac{3}{5} \theta - \frac{2}{5}\rho$. The viscosity coefficient $\mu$ and heat-conductivity coefficient $\kappa$ are defined by $\mu = \mu(1,1)$ and $\kappa = \kappa(1,1)$ as in \eqref{local_viscosity_and_heat_conductivity} by setting $(\rho,\theta) = (1,1)$. For more detailed derivation of the above incompressible Navier-Stokes system, see \cite{BGL1}.

    Various works have been contributed to justify this process since late 1970s. Among them, we shall emphasize the so-called BGL program started by Bardos, Golse, and Levermore since the late 1980s. They aimed to justify Leray's solutions to incompressible Navier-Stokes equations from Diperna-Lions renormalized solutions \cites{BGL1,BGL2}.
After that, Bardos, Golse, Levermore, Lions, Saint-Raymond, Masmoudi then made significant contributions (see e.g., \cites{BGL_2000,Golse_Levermore_2002,Lions_Masmoudi_2001}). Finally Golse and Saint-Raymond obtained the first complete convergence result without any additional compactness assumption in \cite{Golse_Raymond_2004} for cutoff Maxwellian collision kernel and in \cite{Golse_Raymond_2005} for hard cutoff potentials. These results were later extended to include soft potentials by Levermore and Masmoudi \cite{Levermore_Masmoudi_2010}, and to bounded domain by Jiang, Levermore, Masmoudi and Saint-Raymond (see e.g., \cites{Jiang_Masmoudi_2017, Jiang_Levermore_Masmoudi_2010,Masmoudi_Raymond_2003}).

Aside from the uss of renormalized solutions of Boltzmann equation to deduce hydrodynamic limits, numerous works have also been based on the classical solution of Boltzmann equation. The first work of this type is that of Bardos and Ukai \cite{Bardos_Ukai_1991}, in which they proved the global existence of classical solutions of \eqref{perturbed_Boltzmann} for cutoff hard potentials, and established estimates uniform in $\varepsilon$ for $0< \varepsilon<1$. Later Briant, Merino-Aceituno, and Mouhot \cites{Briant_2015,Briant_Merino-Aceituno_Mouhot_2019} used the semigroup approach to prove the incompressible Navier-Stokes limit on torus for cutoff kernels with hard potentials.
For a border class of collision kernels for both cutoff and non-cutoff cases, Jiang, Xu and Zhao proved the incompressible Navier-Stokes-Fourier limit in \cite{Jiang_Xu_Zhao_2018}, using the non-isotropic norm developed in \cites{AMUXY_2011,AMUXY_2012,AMUXY_2012_Anal_Appl} and equivalently in \cite{Gressman_Strain}.
There are also many works based on the Hilbert expansion in the context of classical solutions.
This approach was started from Caflisch and Nishida's work on compressible Euler limit \cites{Caflisch_1980,Nishida_1978}. Guo, Jiang, Jang also used this method combining with some nonlinear estimate to derive the  acoustic limit  \cites{Guo_Jang_Jiang_2009,Guo_Jang_Jiang_2010,Jang_Jiang_2009}. Afterwards, this approach was used for incompressible Navier-Stokes limit in \cites{Guo_CPAM_2006,Masi_Esposito_Lebowitz}. It was later extended to general initial data to include fast acoustic waves by Jiang and Xiong in \cite{Jiang_Xiong}.

In \cite{Golse_Levermore_2002}, Golse and Levermore made the following nice observation (in the {\em Concluding Remarks} of \cite{Golse_Levermore_2002}): they proposed the so-called {\em compressible
 Stokes system}, which is the linearization about a homogeneous state of the compressible Navier-Stokes system. In the short time scale (i.e. $\tau_\varepsilon=1$), as $\varepsilon\rightarrow 0$, the compressible Stokes system converges to the acoustic system. While in the longer time scale, (i.e. $\tau_\varepsilon=\varepsilon$), the limit is the incompressible Stokes equations. In this sense, the (compressible) acoustic system and the incompressible Stokes equations can be unified in one system, which is the compressible Stokes system. 

Golse and Levermore's proposal is on the linear case. Motivated by their idea, in this paper, we start from the scaled general compressible Navier-Stokes, i.e. the diffusion terms are order $O(\varepsilon$. When the time scale $\tau_\varepsilon=1$, the limit will be the compressible Euler system. This is the inviscid limit, which is a well-known analytically hard problem. When the time scale $\tau_\varepsilon=\varepsilon$, the limit will be the incompressible Navier-Stokes equations, which has been well-studied, as mentioned above. The main novelty of this paper is that we propose a new relative entropy approach, which can be analytically proved, as least in the framework of classical solutions. For the clarity of the presentation, we mainly focus on the formal calculation. However, we clearly state the the approach how the rigorous analytical proof could be followed from our formal analysis. The main advantage of this approach is that we could obtain the convergence rate to the incompressible Navier-Stokes equations. This could not be obtained from the previous compactness arguments. We state the main ideas of our appraoch in the following.   

When $\tau_{\varepsilon}=1$, it has been shown in the formal level (see e.g., \cite{BGL1}), using the Chapman-Enskog expansion, that the solution of \eqref{longtime_Boltzmann} can be approximated by $f_{\varepsilon} = M_{\varepsilon} (1 + \varepsilon g_{\varepsilon} + \varepsilon^2 w_{\varepsilon})$,
where $M_{\varepsilon} = \mathcal{M}(\rho_{\varepsilon}, u_{\varepsilon},\theta_{\varepsilon})$ and $(\rho_{\varepsilon}, u_{\varepsilon} ,\theta_{\varepsilon})$ satisfies the following compressible Navier-Stokes system with dissipation of the order $\varepsilon$:
\begin{equation}\label{short_CNSF} \tag{CNS}
    \left\{ \begin{aligned}
        &\partial_t \rho_{\varepsilon} + \nabla_x (\rho_{\varepsilon} u_{\varepsilon}) = 0,\\
        & \rho_{\varepsilon} (\partial_t + u_{\varepsilon} \mdot \nabla_x) u_{\varepsilon}  + \nabla_x ({\rho_{\varepsilon} \theta_{\varepsilon}}) = \varepsilon \nabla_x [ \mu_{\varepsilon}(\rho_{\varepsilon},\theta_{\varepsilon}) \sigma(u_{\varepsilon})],\\
        & \tfrac{3}{2} \rho_{\varepsilon} (\partial_t   + u_{\varepsilon} \mdot \nabla_x )\theta_{\varepsilon} + \rho_{\varepsilon} \theta_{\varepsilon}  \nabla_x \mdot u_{\varepsilon} = \varepsilon  \tfrac{1}{2} \mu_{\varepsilon}(\rho_{\varepsilon},\theta_{\varepsilon}) \sigma(u_{\varepsilon}):\sigma(u_{\varepsilon}) + \tfrac{5}{2} \varepsilon \nabla_x \cdot[\kappa(\rho_{\varepsilon},\theta_{\varepsilon}) \theta_{\varepsilon}],\\
        &(\rho_{\varepsilon}(0,x),\ u_{\varepsilon}(0,x), \ \theta_{\varepsilon}(0,x)) =  (\rho^{\mathrm{in}},\ u^{\mathrm{in}},\ \theta^{\mathrm{in}}).
    \end{aligned}\right.
\end{equation}
Here $\sigma(u)= \nabla_x u + {\nabla_x u}^T  - \frac{2}{3}(\nabla_x \mdot u) \mathrm{I}$ is the stress tensor, 
    and the viscosity coefficient $\mu_{\varepsilon}(\rho,\theta)$ along with the heat-conductivity coefficient $\kappa_{\varepsilon}(\rho,\theta)$ are defined as follows
    \begin{equation}\label{local_viscosity_and_heat_conductivity}
        \begin{aligned}
            \mu(\rho,\theta) = \frac{1}{10} \int_{\mathbb{R}^3_v} A(V) :  \hat{A}(V) \mathcal{M} dv,\\
            \kappa(\rho,\theta) = \frac{2}{15} \int_{\mathbb{R}^3_v} B(V) \cdot \hat{B}(V) \mathcal{M} dv,
        \end{aligned}
    \end{equation}
    where
    \begin{equation}
        A(V) = V\otimes V - \frac{|V|^2}{3}I, \qquad B(V) = V(\frac{|V|^2}{2}- \frac{5}{2}),
    \end{equation}
    with $\displaystyle V= \frac{v-u}{\sqrt{\theta}}$.
    Moreover, $\hat{A} (V)$ and $\hat{B}(V)$ are the unique solutions in $(\text{Ker} \mathcal{L}_{\mathcal{M}}) ^{\perp}$ of the following equations respectively
    \begin{equation}
        \mathcal{L}_{\mathcal{M}} \hat{A} (V) = A(V), \quad  \mathcal{L}_{\mathcal{M}} \hat{B}(V) = B(V),
    \end{equation}
    with
        \begin{equation}
        \mathcal{L}_{\mathcal{M}}(g) := - \frac{1}{\mathcal{M}}(C(\mathcal{M},\mathcal{M}g) + C(\mathcal{M}g,\mathcal{M})).
    \end{equation}

    We note that the compressible Navier-Stokes system \eqref{short_CNSF} is not a limit of the Boltzmann equation as the Knudsen number $\varepsilon \rightarrow 0$ but a second order approximation in the Chapman-Enskog expansion.
    There are also a lot of results on this approximation, we refer to \cites{BGL1,Kawashima_Matsumura_Nishida_1979,Duan_Liu_2021,Liu_Yang_Zhao}  and the references within.

    When Mach number is small (when we set $\tau_\varepsilon = \varepsilon$ as will be assumed throughout), it is expected that the slightly compressible fluids are close to incompressible fluids. 
    Then there is a natural relation between \eqref{short_CNSF} and \eqref{INSF} by studying the fluctuation of $(\rho_\varepsilon,u_\varepsilon,\theta_\varepsilon)$ around $(1,0,1)$. More specifically, one may consider the following  compressible Navier-Stokes system with small Mach number:
    \begin{equation} \label{longtime_CNS}\tag{$\text{CNS}_{\varepsilon}$}
        \left\{ \begin{aligned}
            &\varepsilon\partial_t \rho_{\varepsilon} + \nabla_x (\rho_{\varepsilon} u_{\varepsilon}) = 0,\\
            &\varepsilon\partial_t u_{\varepsilon} + u_{\varepsilon} \mdot \nabla_x u_{\varepsilon}  + \frac{1}{\rho_{\varepsilon}} \nabla_x (\rho_{\varepsilon} \theta_{\varepsilon})  = \varepsilon \frac{1}{\rho_{\varepsilon}} \nabla_x [ \mu_{\varepsilon}(\rho_{\varepsilon},\theta_{\varepsilon}) \sigma(u_{\varepsilon})],\\
            &\varepsilon \partial_t \theta_{\varepsilon}  + u_{\varepsilon} \mdot \nabla_x \theta_{\varepsilon} + \frac{2}{3} \theta_{\varepsilon}  \nabla_x \mdot u_{\varepsilon} = \varepsilon  \frac{1}{\rho_{\varepsilon}} \frac{1}{3} \mu_{\varepsilon}(\rho_{\varepsilon},\theta_{\varepsilon}) \sigma(u_{\varepsilon}):\sigma(u_{\varepsilon}) +  \frac{1}{\rho_{\varepsilon}}\frac{5}{3} \varepsilon \nabla_x \cdot[\kappa(\rho_{\varepsilon},\theta_{\varepsilon}) \theta_{\varepsilon}],\\
            &(\rho_{\varepsilon}(0,x),\ u_{\varepsilon}(0,x), \ \theta_{\varepsilon}(0,x)) = (1+ \varepsilon \tilde{\rho}^{\mathrm{in}},\ \tilde{u}^{\mathrm{in}},\ 1+ \varepsilon \tilde{\theta}^{\mathrm{in}}).
        \end{aligned}\right.
    \end{equation}

    Then the fluctuations $(\tilde{\rho}_{\varepsilon}, \tilde{u}_{\varepsilon},\tilde{\theta}_{\varepsilon})$ of the solutions of \eqref{longtime_CNS}, which satisfying $ (\rho_{\varepsilon}, u_{\varepsilon},\theta_{\varepsilon} )= (1+ \varepsilon \tilde{\rho}_{\varepsilon}, \varepsilon \tilde{u}, 1+ \varepsilon \tilde{\theta}) $, should converge to the solutions of \eqref{INSF} in some strong or weak sense (depending on the initial data and the domain along with proper boundary conditions).
        This process is known as the low Mach number limit, and we refer to \cites{Klainerman_Majda_1981,Klainerman_Majda_1982,Danchin_AJM,Alazard_2006,Ebin,MS_ARMA,DGLM,Masmoudi_2022,CGHJ} for more information.
        We mention that when the domain is considered to be torus, and the initial data are ill-prepared(i.e. the initial data are not required to satisfy the incompressibility and Boussinesq relation), the convergence cannot be strong, due to the acoustic waves generated by initial compression, which will exist in the long time.

        This phenomenon shall also arise when considering the incompressible Navier-Stokes-Fourier limit of the Boltzmann equation, which prevents the strong convergence of $g_{\varepsilon}$ to $g= v \cdot u + \theta(\frac{3}{2} |v|^2 -\frac{5}{2}) $ if the kinetic part of initial data is well-prepared, but the fluid part of initial data is ill-prepared.
        Therefore, when considering the incompressible Navier-Stokes limit of the Boltzmann equation on torus with ill-prepared initial data, we need also to characterize the behavior of acoustic waves in the meantime.

        We deal with this problem by splitting the incompressible Navier-Stokes limit into two steps. The first step is the approximation of Boltzmann equation by compressible Navier-Stokes system with small Mach number, and the second step is the low Mach number limit of the compressible Navier-Stokes system.
        Their relationship is shown in the following diagram. The low Mach number limit process has been discussed in our former paper \cite{CGHJ}. In this paper we focus on the compressible Navier-Stokes approximation for the Boltzmann equation.
        \begin{figure}[H]
            \centering
            \begin{tikzpicture}[>=latex, node distance=2cm]
                \node [draw, rectangle]  (A) at (0,0) {Kinetic Level: };
                \node [draw, rectangle,text width=4cm, align=center] (A_1) at (5,0) {Boltzmann equation\\Kn=Ma=$\varepsilon$};
                \node [draw, rectangle] (B) at (0,-3) {Fluid Level:} ;
                \node [draw, rectangle, text width=3.5cm, align=center]  (C) at(12,-3) { Compressible \\ Navier-Stokes system\\  Ma = $\varepsilon$};
                \node [draw, rectangle, text width=3cm, align=center] (D)  at (5,-3) {Incompressible \\  Naiver-Stokes-Fourier system};

                 \draw[->] (A) -- (B);
                 \draw [->] (A_1) -- (C) node[midway, right,text width=4cm, align=center] {Hydrodynamic \\asymptotic\\ $\varepsilon \rightarrow 0$};
                 \draw [->] (A_1) -- (D) node[midway, left,text width=4cm, align=center] {Hydrodynamic \\limit\\ $\varepsilon \rightarrow 0$};
                 \draw [->] (C) -- (D)node[midway, above,text width=4cm, align=center] {Low Mach number}
                                        node[midway, below,text width=4cm, align=center] {limit $\varepsilon \rightarrow 0$};

            \end{tikzpicture}
            \caption{Relationship between Boltzmann equation, incompressible Navier-Stokes-Fourier system and compressible Navier-Stokes system.}
            \label{fig:tikz_example}
        \end{figure}
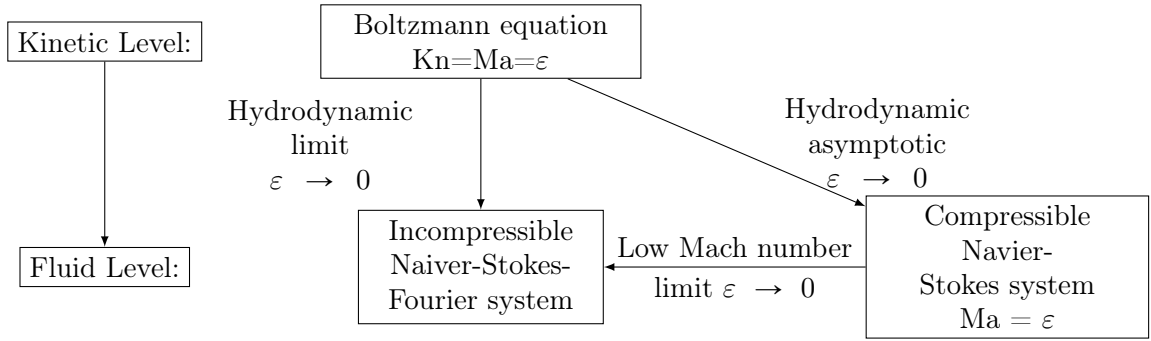

    As illustrated in \eqref{short_CNSF}, the compressible Navier-Stokes-Fourier system derived from Boltzmann equation is {\em not} a limit, since in this setting, the compressible Navier-Stokes-Fourier  system  depends on the Knudsen number itself. So it is not a {\em limit}, but an {\em asymptotic} problem. 
    We shall use the relative entropy to describe the compressible Navier-Stokes approximation of the Boltzmann equation. The relative entropy $H(f|g)$ for $f>0$ and $g>0$
    is defined as
    \begin{equation}\label{RE}
        H(f|g) = \int_{\mathbb{T}^3_x}\int_{\mathbb{R}^3_v} f \log(\frac{f}{g}) - f + g \  dxdv.
    \end{equation}
    The key issue is that in this paper, the ``target'', i.e. $g$ in the relative entropy \eqref{RE} must depends on the Kundsen number $\varepsilon$. More specifically, we set $f$ in \eqref{RE} as the solutions to the scaled Boltzmann equation $f_\varepsilon$,  and $g$ as the local Maxwellian ($M_{\varepsilon}$ defined below) governed by the solutions to the scaled compressible NSF. We calculate the evolution of this scaled relative entropy $\frac{1}{{\varepsilon}^2} H(f_{\varepsilon}| M_{\varepsilon})(t)$ and prove its stability, i.e. the evolution of the scaled relative entropy in any time $t>0$ can be controlled by its initial data. More importantly, we characterize its dissipation rate. It is the first quantitative convergence result for the compressible Navier-Stokes-Fourier system in the sense of relative entropy. This approach avoids any expansion method, such as Chapman-Enskog or Hilbert expansions. The main disadvantage of the expansion methods is it can only construct {\em special} (in the format of expansions) solution of the original Boltzmann equation. On the other hand, the relative entropy method employed here is to directly measure the distance in the sense of relative entropy between {\em any} solutions of the scaled Boltzmann equation and those of compressible NSF. Furthermore, the asymptotic behavior is quantitatively characterized by relative entropy and the relative entropy dissipation is explicit. This is the main novelty of this paper.

    The idea of using the notion of relative entropy for this kind of problems comes from the notion of entropic convergence developed by Bardos, Golse and Levermore in \cite{BGL2}, and on the other hand from Yau's derivation of the hydrodynamic limit of Ginzburg-Landau lattice model \cite{Yau}.
    Latter, Saint-Raymond use relative entropy to deduce the incompressible Euler limit for well-prepared initial data in \cite{Saint-Raymond_2003} and for ill-prepared initial data in \cite{Saint-Raymond_2009}.

    We use the solution $(\rho_{\varepsilon},u_{\varepsilon},\theta_{\varepsilon}) = (1+ \varepsilon \tilde{\rho}_{\varepsilon}, \varepsilon \tilde {u}_{\varepsilon},1+ \varepsilon \tilde{\theta}_{\varepsilon})$ of \eqref{longtime_CNS} to construct a local Maxwellian $M_{\varepsilon} = \mathcal{M}(\rho_{\varepsilon},u_{\varepsilon},\theta_{\varepsilon}) = \mathcal{M} (1+ \varepsilon \tilde{\rho}_{\varepsilon}, \varepsilon \tilde {u}_{\varepsilon},1+ \varepsilon \tilde{\theta}_{\varepsilon}) $,
    and we would like to use the relative entropy $H(f_{\varepsilon}| M_{\varepsilon})$ to measure the difference between the solution of the rescaled Boltzmann equation $\eqref{longtime_Boltzmann}$ and the solution of the compressible Navier-Stokes system $\eqref{longtime_CNS}$.
    More specifically, we derive the time evolution of the relative entropy $\frac{d}{dt} \frac{1}{\varepsilon^2} H(f_{\varepsilon}| M_{\varepsilon})$, which leads to the following modulated entropy inequality
    \begin{equation}
        \begin{aligned}
            \tfrac{1}{{\varepsilon}^2} &H(f_{\varepsilon}| M_{\varepsilon})(t) +\int_0^t \int_{\mathbb{T}^3_x}\left\{\tfrac{1}{\varepsilon^4}D(f_\varepsilon)(s,x) - \tfrac{1}{2} \mu \sigma(\u^b_\varepsilon): \sigma(\u^b_\varepsilon) - \tfrac{5}{2}\kappa (\nabla_x \theta^b_{\varepsilon})^2\right\}\,  \d x\d s \\
            &+\int_0^t \int_{\mathbb{T}^3_x} \tfrac{1}{2}\mu \sigma(\tilde{\u_{\varepsilon}}-\u^b_{\varepsilon}):\sigma(\tilde{\u}_{\varepsilon}-\u^b_{\varepsilon}) + \tfrac{5}{2}  \kappa(\nabla_x \tilde{\theta}_{\varepsilon} -\nabla_x \theta^b_{\varepsilon})^2 \ dxds\\
            &\lesssim \frac{1}{\varepsilon^2} H(f^{\mathrm{in}}_{\varepsilon}| M^{\mathrm{in}}_{\varepsilon}) + \mathcal{O}(\varepsilon).
        \end{aligned}
    \end{equation}
    Here and in the sequel $(\rho^b_{\varepsilon}, \u^b_{\varepsilon},\theta^b_{\varepsilon})$ are moments of perturbation $g_{\varepsilon}$ as defined in \eqref{g_moments}.
    
    Under the assumption of vanishing initial relative entropy
    \begin{equation}
        \frac{1}{\varepsilon^2} H(f^{\mathrm{in}}_{\varepsilon}| M_{\varepsilon}^{\mathrm{in}}) \rightarrow 0  \ \ \text{as } \varepsilon \rightarrow 0,
    \end{equation}
    for each $t > 0$, we obtain the following asymptotics:
    \begin{enumerate}[label=\textbullet]
        \item Asymptotic of the entropy dissipation rate
            \begin{equation}
               \lim_{\varepsilon \rightarrow 0} \left[\int_0^t \int_{\mathbb{T}^3_x}\frac{1}{\varepsilon^4}D(f_\varepsilon)(s,x) - \frac{1}{2} \mu \sigma(u^b_{\varepsilon}): \sigma(u^b_{\varepsilon}) - \frac{5}{2}\kappa (\nabla_x \theta^b_{\varepsilon})^2\  dxds \right] = 0;
            \end{equation}
        \item Asymptotic of the momentum and energy flux
            \begin{equation}
                \lim_{\varepsilon \rightarrow 0} \left[ \int_0^t \int_{\mathbb{T}^3_x} \frac{1}{2}\mu \sigma(\tilde{\u}_{\varepsilon}-\u^b_{\varepsilon}):\sigma(\tilde{\u}_{\varepsilon}-\u^b_{\varepsilon}) + \frac{5}{2}  \kappa(\nabla_x \tilde{\theta}_{\varepsilon} -\nabla_x \theta^b_{\varepsilon})^2 \ dxds \right]=0;
            \end{equation}
        \item Asymptotic of the relative entropy
            \begin{equation}
                \lim_{\varepsilon \rightarrow 0}\frac{1}{{\varepsilon}^2} H(f_{\varepsilon}| M_{\varepsilon})(t)= 0.
            \end{equation}
    \end{enumerate}
   We state our main theorem as follows:
    \begin{theorem}\label{main}
        Let $f_{\varepsilon}$  be of the form $f_{\varepsilon} = M + \varepsilon M g_{\varepsilon} $, and $\{f_{\varepsilon}\}$ is a family of solutions to \eqref{longtime_Boltzmann} with $\tau_\varepsilon=\varepsilon$ satisfying the assumptions \eqref{conservation_of_mass},\eqref{conservation_of_momentum}, \eqref{conservation_of_energy} and \eqref{entropy_inequality}.
        Let $(\rho_{\varepsilon}, u_{\varepsilon},\theta_{\varepsilon})$ be of the form $(\rho_{\varepsilon},u_{\varepsilon},\theta_{\varepsilon}) = (1+ \tilde{\rho}_{\varepsilon},\varepsilon \tilde{u}_{\varepsilon},1+ \varepsilon \tilde{\theta}_{\varepsilon})$, and $\{(\rho_{\varepsilon},u_{\varepsilon},\theta_{\varepsilon})\}$ is a family of solutions to \eqref{longtime_CNS}.
        Then we have, for $t \geqslant 0$,
        \begin{equation}\label{main_estimate}
            \frac{1}{{\varepsilon}^2} H(f_{\varepsilon}| M_{\varepsilon})(t)  \leqslant \frac{1}{\varepsilon^2}H(f^{\mathrm{in}}_{\varepsilon}| M^{\mathrm{in}}_{\varepsilon} ) + \mathcal{O}(\varepsilon),
        \end{equation}
        \begin{equation}\label{entropy_dissipation_bound}
           \int_0^t \int_{\mathbb{T}^3_x}\frac{1}{\varepsilon^4}D(f_\varepsilon)(s,x) - \frac{1}{2} \mu \sigma(\u^b_\varepsilon): \sigma(\u^b_\varepsilon) - \frac{5}{2}\kappa (\nabla_x \theta^b_{\varepsilon})^2\  dxds \lesssim \frac{1}{\varepsilon^2} H(f_{\varepsilon}^{\mathrm{in}}|M_{\varepsilon}^{\mathrm{in}}) + \mathcal{O}(\varepsilon),
         \end{equation}
         \begin{equation}\label{energy_momentum_flux_bound}
             \int_0^t \int_{\mathbb{T}^3_x} \frac{1}{2}\mu \sigma(\tilde{\u}_{\varepsilon}-\u^b_{\varepsilon}):\sigma(\tilde{\u}_{\varepsilon}-\u^b_{\varepsilon}) + \frac{5}{2}  \kappa(\nabla_x \tilde{\theta} -\nabla_x \theta^b_{\varepsilon})^2 \ dxds \lesssim \frac{1}{\varepsilon^2} H(f_{\varepsilon}^{\mathrm{in}}|M_{\varepsilon}^{\mathrm{in}}) + \mathcal{O}(\varepsilon),
        \end{equation}
        where
        \begin{equation}
            M_{\varepsilon}^{\mathrm{in}} = M(1+ \varepsilon \tilde\rho^{\mathrm{in}}, \varepsilon \tilde{u}^{\mathrm{in}},1 + \varepsilon \tilde{\theta}^{\mathrm{in}}).
        \end{equation}
    \end{theorem}
\begin{remark}
  We remark that although the above theorem is in the formal level, its proof is indeed provide a clear way to the later rigorous proof. If we work in the framework of the classical solutions, all the corresponding existence results are available, for the Boltzmann equations, compressible and incompressible Navier-Stokes equations, at least near the constant states. Because the scalings themselves are near the constant states, so in this sense, the formal analysis is not far from the rigorous proof. 

    In fact, the estimates \eqref{main_estimate}, \eqref{entropy_dissipation_bound} and \eqref{energy_momentum_flux_bound} hold for $0 \leqslant t \leqslant T^*$, where $T^*$ is determined by the minimal lifespan of the solutions to \eqref{longtime_Boltzmann} and \eqref{longtime_CNS}. In particular, if  for all $ \varepsilon > 0$ both $f_\varepsilon$ and $(\rho_\varepsilon, u_\varepsilon, \theta_\varepsilon)$ exist globally in time, then these estimates also hold globally. Furthermore, the validity of the $\mathcal{O}(\varepsilon)$ residual term relies on uniform in $\varepsilon$ estimates for the solutions. As we are working within a perturbation framework, such existence of global solutions of \eqref{longtime_CNS} are provided in \cites{Danchin_Inventiones,Chen-Gui-Jiang}, while the required uniform bounds for $(\rho_\varepsilon, u_{\varepsilon}, \theta_{\varepsilon})$ can be found in \cite{CGHJ}. Moreover, the global existence and uniform in $\varepsilon$ estimates for the solution $f_{\varepsilon}$ of \eqref{longtime_Boltzmann} can be found in \cite{Jiang_Xu_Zhao_2018}.
\end{remark}
    In the next section, we shall prove, at a formal level, the above estimate by calculating the evolution of the relative entropy. In section \ref{SS_REC}, we show that the unsigned terms will be controlled by the relative entropy itself. In section \ref{SS_EDC}  we show that the negative terms will be controlled by the entropy dissipation.
    And in section \ref{SS_Conclusion} we conclude the estimate \eqref{main_estimate}, \eqref{entropy_dissipation_bound}, \eqref{energy_momentum_flux_bound} by Grönwall's inequality.
    \section{Compressible Navier-Stokes Approximation}

    Suppose that $f_\varepsilon= f_{\varepsilon}(t,x,v)$ is of the form $f_{\varepsilon} = M + \varepsilon M g_{\varepsilon}$, and $g_{\varepsilon}$ is in some sense a family of  solutions of \eqref{perturbed_Boltzmann}. Also, suppose that
    $f_{\varepsilon}$ satisfies the local conservation of density, momentum, and energy as follows:
        \begin{align}
            &\partial_t \int_{\mathbb{R}^3_v} f_{\varepsilon} dv + \nabla_x \cdot  \int_{\mathbb{R}^3_v} f_{\varepsilon}dv = 0, \label{conservation_of_mass} \\
            &\partial_t \int_{\mathbb{R}^3_v} vf_{\varepsilon} dv + \nabla_x \cdot \int_{\mathbb{R}^3_v} vf_{\varepsilon}dv =0,\label{conservation_of_momentum}\\
            &\partial_t \int_{\mathbb{R}^3_v} |v|^2 f_{\varepsilon} dv + \nabla_x \cdot \int_{\mathbb{R}^3_v} |v|^2f_{\varepsilon}dv=0.\label{conservation_of_energy}\\
        \end{align}
    Also, $f_{\varepsilon}$ should satisfy the following relative entropy inequality:
    \begin{equation}\label{entropy_inequality}
        H(f_{\varepsilon}| M)(t) + \frac{1}{\varepsilon^2}\int_0^t \int_{\mathbb{T}^3_x} D(f_\varepsilon)(s,x) \ ds dx \leqslant H(f^{\mathrm{in}}_{\varepsilon}| M ),
    \end{equation}
    where the relative entropy $H(f|g)$ for $f>0$ and $g>0$ is defined as
    \begin{equation}
        H(f|g) = \int_{\mathbb{T}^3_x}\int_{\mathbb{R}^3_v} f \log(\frac{f}{g}) - f + g \  dxdv,
    \end{equation}
    and $D(f)$ is defined by
    \begin{equation}
        D(f) = \frac{1}{4} \int_{\mathbb{R}^3_v}\int_{\mathbb{R}^3_v}\int_{\mathbb{S}^{N-1}}( f^{\prime} f^{\prime}_1 -f f_1) \log \frac{f^{\prime} f^{\prime}_1}{f f_1} b \ dv dv_1 d\sigma.
    \end{equation}

    To prove the estimate \eqref{main_estimate}, note that
    \begin{equation} \label{entropy_relative_entropy_relation}
        H(f_{\varepsilon}| M) = H(f_{\varepsilon}| M_{\varepsilon}) +   \int_{\mathbb{T}^3_x} \int_{\mathbb{R}^3_v} f_{\varepsilon} \log \frac{M_{\varepsilon}}{M} - M_{\varepsilon} + M \ dx dv.
    \end{equation}
    Combing \eqref{entropy_inequality} and \eqref{entropy_relative_entropy_relation} we have
    \begin{equation}\label{relative_entropy}
        \begin{aligned}
        \frac{1}{{\varepsilon}^2} H(f_{\varepsilon}| M_{\varepsilon})(t)
        + \frac{1}{\varepsilon^2}\int^t_0 \int_{\mathbb{T}^3_x} \int_{\mathbb{R}^3_v}  \frac{d}{dt}(f_{\varepsilon} \log \frac{M_{\varepsilon}}{M} - M_{\varepsilon} + M) \ dx dv ds\\
         +\frac{1}{\varepsilon^4}\int_0^t \int_{\mathbb{T}^3_x} D(f_\varepsilon)(s,x) \ ds dx
          \leqslant \frac{1}{\varepsilon^2}H(f^{\mathrm{in}}_{\varepsilon}| M^{\mathrm{in}}_{\varepsilon} ).
        \end{aligned}
    \end{equation}
    Clearly,
    \begin{equation}
        \int^t_0 \int_{\mathbb{T}^3_x} \int_{\mathbb{R}^3_v}\frac{d}{dt} M_{\varepsilon} \ dx dv ds = \int^t_0 \int_{\mathbb{T}^3_x} \frac{d}{dt} \rho_{\varepsilon} \ dx ds=0.
    \end{equation}
    We focus on the rest term:
    \begin{equation}\label{temp_dt1}
        \begin{aligned}
            \int^t_0 \int_{\mathbb{T}^3_x} \int_{\mathbb{R}^3_v}  \frac{d}{dt}(f_{\varepsilon} \log \frac{M_{\varepsilon}}{M} ) \ dx dv ds
            =  \int^t_0 \int_{\mathbb{T}^3_x} \int_{\mathbb{R}^3_v} \partial_t (f_{\varepsilon})\log \frac{M_{\varepsilon}}{M} + f_{\varepsilon} \partial_t \log M_{\varepsilon} \ dx dv ds. \\
        \end{aligned}
    \end{equation}
    Using \eqref{longtime_Boltzmann} and noticing that $\log(\frac{M_{\varepsilon}}{M})$ is a linear combination of $1,v \text{ and } |v|^2$, we then have
    \begin{equation}\label{temp_dt2}
        \begin{aligned}
           \int^t_0 \int_{\mathbb{T}^3_x} \int_{\mathbb{R}^3_v} \partial_t (f_{\varepsilon})\log \frac{M_{\varepsilon}}{M} \ dx dv ds
             =  \int^t_0 \int_{\mathbb{T}^3_x} \int_{\mathbb{R}^3_v} - \frac{1}{\varepsilon} v\mdot \nabla_x (f_{\varepsilon})\log \frac{M_{\varepsilon}}{M}\ dx dv ds.
        \end{aligned}
    \end{equation}
    Combing \eqref{temp_dt1} and \eqref{temp_dt2} and using integration by parts leads to
    \begin{equation}\label{temp_dt3}
        \begin{aligned}
            \int^t_0 \int_{\mathbb{T}^3_x} \int_{\mathbb{R}^3_v}  \frac{d}{dt}(f_{\varepsilon} \log \frac{M_{\varepsilon}}{M} ) \ dx dv ds
            = \int^t_0 \int_{\mathbb{T}^3_x} \int_{\mathbb{R}^3_v} f_{\varepsilon} \left(\partial_t + \frac{1}{\varepsilon} v\mdot \nabla_x \right) \log M_{\varepsilon} \ dx dv ds .
        \end{aligned}
    \end{equation}
    To proceed, we need the following observation.
    \begin{lemma} \label{observation_lemma}
    For $\rho >0$ and  $\theta >0$, we have
    \begin{equation}\label{Euler_operator}
        \begin{aligned}
        v \mdot \nabla_x\log \mathcal{M}(\rho,u,\theta) &=
        \left(
        \begin{pmatrix}
            u\mdot \nabla_x \rho + \rho \nabla_x \mdot u\\
            u \mdot \nabla_x u  + \frac{\theta}{\rho} \nabla_x \rho + \nabla_x \theta\\
            u \mdot \nabla_x \theta + \frac{2}{3} \theta  \nabla_x \mdot u
        \end{pmatrix}
        \raisebox{-2.4ex}{,}
        \begin{pmatrix}
            1/ \rho \\
            V / \sqrt{\theta}\\
            \frac{1}{\theta}(\frac{|V|^2}{2} -\frac{3}{2})
        \end{pmatrix}
        \right)\\
        &+A(V):\nabla_x u + B(V)\cdot \frac{\nabla_x \theta}{\sqrt{\theta}}.
        \end{aligned}
    \end{equation}
    \end{lemma}

    \begin{remark}
        This equality shows a connection between the Boltzmann equation and the fluid equation(Compressible Euler, Compressible Navier-Stokes), as the convection terms and the pressure terms in the fluid equation arise  naturally from this formulation.
        
        Also, if we linearize above equality around the global Maxwellian $\mathcal{M}(1,0,1)$, we shall obatin the following frequently used equality, which indicates the connection to acoustic system,
        \begin{equation}\label{v_dot_nabla_g}
            \begin{aligned}
            v \mdot \nabla_x g = \left(\begin{pmatrix}
                \nabla \!\cdot\! u\\
                \nabla (\rho + \theta)\\
                \frac{2}{3} \nabla \!\cdot\! u
            \end{pmatrix}
            \raisebox{-2.4ex}{,}
        \begin{pmatrix}
            1 \\
            v \\
            (\frac{|v|^2}{2} -\frac{3}{2})
        \end{pmatrix}\right)
        +A(v):\nabla_x u + B(v)\cdot {\nabla_x \theta},
            \end{aligned}
        \end{equation}
        where $g = \rho + u \cdot v + (\frac{|v|^2}{2} -\frac{3}{2})\theta$.
    \end{remark}
    Now using \eqref{temp_dt1}, \eqref{temp_dt3} and \eqref{Euler_operator}, we obtain
    \begin{equation}\label{temp_result_1}
        \begin{aligned}
            \int^t_0 \int_{\mathbb{T}^3_x} \int_{\mathbb{R}^3_v}  &\frac{d}{dt}(f_{\varepsilon} \log \frac{M_{\varepsilon}}{M} ) \ dx dv ds\\
            &=
            \frac{1}{\varepsilon}\int^t_0 \int_{\mathbb{T}^3_x} \int_{\mathbb{R}^3_v} \left(\begin{pmatrix}
                \varepsilon \partial_t \rho_{\varepsilon}+ u_{\varepsilon}\mdot \nabla_x \rho_{\varepsilon} + \rho_{\varepsilon} \nabla_x \mdot u_{\varepsilon}\\
                \varepsilon \partial_t u_{\varepsilon} + u_{\varepsilon} \mdot \nabla_x u_{\varepsilon}  + \frac{\theta_{\varepsilon}}{\rho_{\varepsilon}} \nabla_x \rho_{\varepsilon} + \nabla_x \theta_{\varepsilon}\\
                \varepsilon \partial_t \theta_{\varepsilon} + u_{\varepsilon} \mdot \nabla_x \theta_{\varepsilon} + \frac{2}{3} \theta_{\varepsilon}  \nabla_x \mdot u_{\varepsilon}
            \end{pmatrix}
            \raisebox{-2.4ex}{,}
            \begin{pmatrix}
                1/ \rho_{\varepsilon} \\
                V_\varepsilon / \sqrt{\theta_{\varepsilon}}\\
                \frac{1}{\theta_{\varepsilon}}(\frac{|V_\varepsilon|^2}{2} -\frac{3}{2})
            \end{pmatrix}\right)f_{\varepsilon}\ dx dv ds\\
            &+\frac{1}{\varepsilon}\int^t_0 \int_{\mathbb{T}^3_x} \int_{\mathbb{R}^3_v} A(V_\varepsilon):\nabla_x u_{\varepsilon} f_{\varepsilon} + B(V_\varepsilon)\cdot \frac{\nabla_x \theta_{\varepsilon}}{\sqrt{\theta_{\varepsilon}}} f_{\varepsilon}\ dx dv ds.
        \end{aligned}
    \end{equation}
    Here $V_{\varepsilon}= \frac{v-u_\varepsilon}{\sqrt{\theta_{\varepsilon}}}$. 
    Combing \eqref{relative_entropy}, \eqref{temp_result_1} and \eqref{longtime_CNS} yields that
\begin{equation}\label{temp_result_2}
    \begin{aligned}
    &\frac{1}{{\varepsilon}^2} H(f_{\varepsilon}| M_{\varepsilon})(t) +\frac{1}{\varepsilon^4}\int_0^t \int_{\mathbb{T}^3_x} D(f_\varepsilon)(s,x)\ ds dx\\
    &+ \frac{1}{\varepsilon^3}\int^t_0 \int_{\mathbb{T}^3_x} \int_{\mathbb{R}^3_v} \varepsilon \nabla_x \mdot [ \mu(\rho_{\varepsilon},\theta_{\varepsilon}) \sigma(u_{\varepsilon})] \mdot \frac{1}{\rho_{\varepsilon}}\frac{V_\varepsilon} {\sqrt{\theta_{\varepsilon}}} f_{\varepsilon} \\
    & + \frac{1}{\varepsilon^3}\int^t_0 \int_{\mathbb{T}^3_x} \int_{\mathbb{R}^3_v}\frac{1}{\rho_{\varepsilon}}\left(\varepsilon \frac{1}{3} \mu(\rho_{\varepsilon},\theta_{\varepsilon}) \sigma(u_{\varepsilon}):\sigma(u_{\varepsilon}) + \frac{5}{3} \varepsilon \nabla_x \mdot [\kappa(\rho_{\varepsilon},\theta_{\varepsilon}) \nabla_x \theta_{\varepsilon}]\right) f_{\varepsilon}
    \frac{1}{\theta_{\varepsilon}}(\frac{|V_\varepsilon|^2}{2} -\frac{3}{2}) \ dx dv ds
    \\
    &+\frac{1}{\varepsilon^3} \int^t_0 \int_{\mathbb{T}^3_x} \int_{\mathbb{R}^3_v} A(V_\varepsilon):\nabla_x u_{\varepsilon} f_{\varepsilon} + B(V_\varepsilon)\cdot \frac{\nabla_x \theta_{\varepsilon}}{\sqrt{\theta_{\varepsilon}}} f_{\varepsilon}  \ dx dv ds \leqslant \frac{1}{\varepsilon^2} H(f^{\mathrm{in}}_{\varepsilon}| M^{\mathrm{in}}_{\varepsilon} ).
    \end{aligned}
\end{equation}

Now we set $f_{\varepsilon} = M(1+ \varepsilon g_{\varepsilon})$ and $(\rho_{\varepsilon},\ u_{\varepsilon},\ \theta_{\varepsilon} )= (1 + \varepsilon\tilde{\rho}_{\varepsilon},\ \varepsilon \tilde{u}_{\varepsilon},\ 1+ \varepsilon \tilde{\theta}_{\varepsilon})$, and we denote
\begin{equation}\label{g_moments}
    \begin{aligned}
        \rho^b_{\varepsilon} = \langle g_{\varepsilon} \rangle,\
        u^b_{\varepsilon} =\langle  vg_{\varepsilon} \rangle,\
        \theta^b_{\varepsilon} = \langle  \frac{|v|^2 -3}{3} g_{\varepsilon} \rangle.
    \end{aligned}
\end{equation}
For notational simplicity, we will drop the subscript $\varepsilon$ in $V_{\varepsilon}$, $(\rho^b_{\varepsilon},u^b_{\varepsilon},\theta^b_{\varepsilon})$, $(\tilde{\rho}_{\varepsilon},\tilde{u}_{\varepsilon},\tilde{\theta}_{\varepsilon})$ and $(\rho_{\varepsilon},u_{\varepsilon},\theta_{\varepsilon})$ in the sequel.

Simple calculation then shows that
\begin{equation}\label{V_f}
    \begin{aligned}
        \int_{\mathbb{R}^3_v} \frac{V} {\sqrt{\theta}} f_{\varepsilon} dv =\int_{\mathbb{R}^3_v}\frac{v-\varepsilon\tilde{u}} {\theta} M(1+ \varepsilon g_{\varepsilon}) dv
        =\varepsilon\frac{( u^b -\tilde{u})}{\theta} - \varepsilon^2 \frac{\rho^b\tilde{u}}{\theta},
    \end{aligned}
\end{equation}

\begin{equation}\label{V_square_f}
    \begin{aligned}
        \int_{\mathbb{R}^3_v}  \frac{1}{\theta}(\frac{|V|^2}{2} -\frac{3}{2}) f_{\varepsilon} dv
        &=\int_{\mathbb{R}^3_v} \left(\frac{1}{\theta} \frac{(v-\varepsilon \tilde{u})^2}{2\theta} -\frac{3}{2\theta}\right) (1+ \varepsilon g_{\varepsilon})M dv\\
        &=\varepsilon \frac{1}{\theta^2} \frac{3}{2} (\theta^b -\theta) + \varepsilon^2 \frac{3}{2} \frac{\tilde{\theta \rho^b}}{\theta^2}
         - \varepsilon^2 \frac{\tilde{u} \cdot u^b}{\theta^2}+ \varepsilon^3 \frac{\tilde{u}^2 \rho^b}{2\theta^2},
    \end{aligned}
\end{equation}

\begin{equation}\label{AV_f}
    \begin{aligned}
        \int A(V) f_{\varepsilon} dv &= \frac{1}{\theta}\left(\varepsilon^2 (\tilde{u} \otimes \tilde{u} - \frac{\tilde{u}^2}{3}I) - \varepsilon^2 ( \tilde{u} \otimes u^b + u^b \otimes \tilde{u}) + \varepsilon^2 \frac{2\tilde{u}\cdot u^b}{3}I + \varepsilon^3 \rho^b  (\tilde{u} \otimes \tilde{u} - \frac{\tilde{u}^2}{3}I) \right) \\
        &+ \varepsilon \frac{1}{\theta}\langle A(v), g_{\varepsilon} \rangle.
    \end{aligned}
\end{equation}

\begin{equation}\label{BV_f}
    \begin{aligned}
       \int B(V) f_{\varepsilon} \frac{1}{\sqrt{\theta}} dv = \varepsilon^2 \frac{5}{2\theta^2} (\tilde{\theta} \tilde{u} - \tilde{\theta}u^b - \tilde{u} \theta^b ) + \varepsilon \frac{1}{\theta^2}\langle B(v), g_{\varepsilon} \rangle - \varepsilon^2 \frac{1}{\theta^2} \tilde{u} \langle A(v), g_{\varepsilon} \rangle \\
       + \varepsilon^3 ( -\frac{\tilde{u}^2\tilde{u}}{2\theta^2} + \frac{\tilde{u}^2 u^b}{2\theta^2}+ \frac{5\rho^b \tilde{\theta}\tilde{u}}{2\theta^2} + \frac{\tilde{u}\otimes \tilde{u} \cdot u^b}{\theta^2}) -\varepsilon^4\frac{\tilde{u}^2\tilde{u}\rho^b}{2\theta^2}.
    \end{aligned}
\end{equation}
Substituting \eqref{V_f}, \eqref{V_square_f}, \eqref{AV_f} and \eqref{BV_f} into \eqref{temp_result_2} yields that
\begin{equation}\label{temp_result_before_expand_AB}
    \begin{aligned}
    &\frac{1}{{\varepsilon}^2} H(f_{\varepsilon}| M_{\varepsilon})(t) +\frac{1}{\varepsilon^4}\int_0^t \int_{\mathbb{T}^3_x} D(f_\varepsilon)(s,x) \ ds dx\\
    &+ \int^t_0 \int_{\mathbb{T}^3_x}  \nabla_x \mdot [ \mu(\rho,\theta) \sigma(\tilde{u})] \mdot \frac{1}{\rho} \frac{u^b-\tilde{u}}{\theta} \ dx ds+ R_1\\
    & + \int^t_0 \int_{\mathbb{T}^3_x} \frac{1}{\rho}\left(\frac{5}{3}  \nabla_x \mdot [\kappa(\rho,\theta) \nabla_x \tilde{\theta}]\right) \frac{1}{\theta^2} \frac{3}{2}(\theta^b -\tilde{\theta}) \ dx ds  + R_2\\
    &+ \int^t_0 \int_{\mathbb{T}^3_x} \frac{1}{\theta} \left[ (\tilde{u}-u^b)\otimes(\tilde{u}-u^b) - \frac{1}{3} (\tilde{u} - u^b)^2 I -u^b \otimes u^b   + \frac{(u^b)^2}{3}I + \frac{1}{\varepsilon} \langle A(v), g_{\varepsilon}\rangle \right] : \nabla_x \tilde{u} \  dx ds +R_3 \\
    &+ \int^t_0 \int_{\mathbb{T}^3_x} \left[\frac{5}{2} \frac{1}{\theta^2} (\tilde{u}-u^b)(\tilde{\theta}-\theta^b) - \frac{5}{2} \frac{1}{\theta^2}u^b \theta^b + \frac{1}{\varepsilon} \frac{1}{\theta^2}\langle B(v),g_{\varepsilon} \rangle \right] \cdot \nabla_x \tilde{\theta}  \ dx ds + R_4 \leqslant \frac{1}{\varepsilon^2} H(f^{\mathrm{in}}_{\varepsilon}| M^{\mathrm{in}}_{\varepsilon} ),
    \end{aligned}
\end{equation}
where
\begin{equation}
    \begin{aligned}
        R_1 &= \varepsilon \int^t_0 \int_{\mathbb{T}^3_x} \frac{\rho^b \tilde{u}}{\rho\theta} \ dxds,\\
        R_2 &= \varepsilon \int^t_0 \int_{\mathbb{T}^3_x} \frac{1}{\rho} \frac{1}{3} \mu(\rho,\theta) \sigma(\tilde{u}):\sigma(\tilde{u})
        \left(  \frac{1}{\theta^2} \frac{3}{2} (\theta^b -\theta) + \varepsilon \frac{3}{2} \frac{\tilde{\theta \rho^b}}{\theta^2}
        - \varepsilon \frac{\tilde{u} \cdot u^b}{\theta^2}+ \varepsilon^2 \frac{\tilde{u}^2 \rho^b}{2\theta^2}\right)\ dxds\\
            &+ \varepsilon \int^t_0 \int_{\mathbb{T}^3_x}  \frac{1}{\rho}\left(\frac{5}{3}  \nabla_x \mdot [\kappa(\rho,\theta) \nabla_x \tilde{\theta}]\right)
            \left(  \frac{3}{2} \frac{\tilde{\theta \rho^b}}{\theta^2}
            -  \frac{\tilde{u} \cdot u^b}{\theta^2}+ \varepsilon \frac{\tilde{u}^2 \rho^b}{2\theta^2}\right) \ dxds ,\\
        R_3 &= \varepsilon \int^t_0 \int_{\mathbb{T}^3_x} \frac{\rho^b}{\theta}(\tilde{u}\otimes \tilde{u} -\frac{|\tilde{u}|^2}{3}I) : \nabla_x \tilde{u} \ dxds,\\
        R_4 &=\varepsilon \int^t_0 \int_{\mathbb{T}^3_x}   -\frac{\tilde{u}^2\tilde{u}}{2\theta^2} + \frac{\tilde{u}^2 u^b}{2\theta^2}+ \frac{5\rho^b \tilde{\theta}\tilde{u}}{2\theta^2} + \frac{\tilde{u}\otimes \tilde{u} \cdot u^b}{\theta^2} -\varepsilon\frac{\tilde{u}^2\tilde{u}\rho^b}{2\theta^2} \ dxds
        +\int^t_0 \int_{\mathbb{T}^3_x} -\langle A(v),g_{\varepsilon} \rangle dxds .
    \end{aligned}
\end{equation}
We now need to calculate the terms $\frac{1}{\varepsilon} \langle A(v),g_{\varepsilon} \rangle$ and $\frac{1}{\varepsilon}\langle B(v), g_{\varepsilon} \rangle$. Indeed, we have the following lemma:

\begin{lemma}\label{AvBv_asymptotics}
    \begin{equation}\label{Av_g}
        \begin{aligned}
        \frac{1}{\varepsilon}\langle A(v), g_{\varepsilon} \rangle &=-\mu \sigma(u^b) +\left(u^b \otimes u^b - \frac{(u^b)^2}{3}I\right) + R_A,\\
        \frac{1}{\varepsilon}\langle B(v), g_{\varepsilon} \rangle &=-\frac{5}{2}\kappa \nabla_x \theta^b +\frac{5}{2}u^b \theta^b  + R_B,
        \end{aligned}
    \end{equation}
    where
    \begin{equation}
        \begin{aligned}
            R_A =  2\langle \hat{A}(v), Q(\Loth g_{\varepsilon},\Lpre g_{\varepsilon})  \rangle + \langle \hat{A}(v), Q(\Loth g_{\varepsilon},\Loth g_{\varepsilon})  \rangle
            -\langle \hat{A}(v),v\mdot \nabla_x \Loth g_{\varepsilon} \rangle
            + \varepsilon \langle  \hat{A}(v), -\partial_t \Loth g_{\varepsilon} \rangle,\\
            R_B= 2\langle \hat{B}(v), Q(\Loth g_{\varepsilon},\Lpre g_{\varepsilon})  \rangle + \langle \hat{B}(v), Q(\Loth g_{\varepsilon},\Loth g_{\varepsilon})  \rangle
            -\langle \hat{B}(v),v\mdot \nabla_x \Loth g_{\varepsilon} \rangle
            + \varepsilon \langle  \hat{B}(v), -\partial_t \Loth g_{\varepsilon} \rangle.
        \end{aligned}
    \end{equation}
\end{lemma}

Lemma \ref{AvBv_asymptotics} and \eqref{temp_result_before_expand_AB} then leads to

    \begin{equation}\label{temp_relative_entropy}
        \begin{aligned}
        &\frac{1}{{\varepsilon}^2} H(f_{\varepsilon}| M_{\varepsilon})(t) +\frac{1}{\varepsilon^4}\int_0^t \int_{\mathbb{T}^3_x} D(f_\varepsilon)(s,x)\ ds dx + \tilde{R} \\
        &+ \int^t_0 \int_{\mathbb{T}^3_x}  \nabla_x [ \mu(\rho,\theta) \sigma(\tilde{u})] \frac{1}{\rho} \frac{u^b-\tilde{u}}{\theta} - \frac{1}{\theta} \mu \sigma(u^b):\nabla_x \tilde{u}  \ dx ds  \\
        & + \int^t_0 \int_{\mathbb{T}^3_x} \frac{1}{\rho}\left(\frac{5}{2}  \nabla_x [\kappa(\rho,\theta) \nabla_x \tilde{\theta}]\right) \frac{1}{\theta^2} (\theta^b -\tilde{\theta}) -\frac{1}{\theta^2}\frac{5}{2}\kappa \nabla_x \theta^b \cdot \nabla_x \tilde{\theta} \ dx ds  \\
        &+ \int^t_0 \int_{\mathbb{T}^3_x} \frac{1}{\theta} \left[ (\tilde{u}-u^b)\otimes(\tilde{u}-u^b)  - \frac{1}{3} (\tilde{u} - u^b)^2 I   \right] : \nabla_x \tilde{u} \  dx ds  \\
        &+ \int^t_0 \int_{\mathbb{T}^3_x} \left[\frac{5}{2} \frac{1}{\theta^2} (\tilde{u}-u^b)(\tilde{\theta}-\theta^b) \right] \cdot \nabla_x \tilde{\theta}  \ dx ds    \\
        &\leqslant \frac{1}{\varepsilon^2} H(f^{\mathrm{in}}_{\varepsilon}| M^{\mathrm{in}}_{\varepsilon} ),
        \end{aligned}
    \end{equation}
where
\begin{equation}
    \begin{aligned}
        \tilde{R} = R_1 + R_2 + R_3 + R_4 + \int^t_0 \int_{\mathbb{T}^3_x} R_A : \frac{\nabla_x \tilde{u}}{\theta} + R_B \cdot \frac{\nabla_x \tilde{\theta}}{\theta^2} \ dxds.
    \end{aligned}
\end{equation}
And we denote
\begin{align}
    & \text{\Rmnum{1}}=\int^t_0 \int_{\mathbb{T}^3_x}  \nabla_x [ \mu(\rho,\theta) \sigma(\tilde{u})] \frac{1}{\rho} \frac{u^b-\tilde{u}}{\theta} - \frac{1}{\theta} \mu \sigma(u^b):\nabla_x \tilde{u}  \ dx ds  \label{viscous_term},\\
    & \text{\Rmnum{2}}= \int^t_0 \int_{\mathbb{T}^3_x} \frac{1}{\rho}\left(\frac{5}{2}  \nabla_x [\kappa(\rho,\theta) \nabla_x \tilde{\theta}]\right) \frac{1}{\theta^2} (\theta^b -\tilde{\theta}) -\frac{1}{\theta^2}\frac{5}{2}\kappa \nabla_x \theta^b \cdot \nabla_x \tilde{\theta} \ dx ds \label{heat_conduct_term} ,\\
    & \text{\Rmnum{3}}= \int^t_0 \int_{\mathbb{T}^3_x} \frac{1}{\theta} \left[ (\tilde{u}-u^b)\otimes(\tilde{u}-u^b)  - \frac{1}{3} (\tilde{u} - u^b)^2 I   \right] : \nabla_x \tilde{u} \  dx ds   \label{u_convection_term}, \\
    & \text{\Rmnum{4}}= \int^t_0 \int_{\mathbb{T}^3_x} \left[\frac{5}{2} \frac{1}{\theta^2} (\tilde{u}-u^b)(\tilde{\theta}-\theta^b) \right] \cdot \nabla_x \tilde{\theta}  \ dx ds   \label{theta_convection_term} .\\
    \end{align}
For the term \Rmnum{1}, by symmetry of $\sigma(\tilde{u})$ and integration by parts, we have
\begin{equation}\label{integration_by_part_u}
    \begin{aligned}
        \int^t_0 \int_{\mathbb{T}^3_x} &\nabla_x \mdot [ \mu(\rho,\theta) \sigma(\tilde{u})] \mdot \frac{1}{\rho} \frac{u^b-\tilde{u}}{\theta} - \frac{1}{\theta} \mu \sigma(u^b):\nabla_x \tilde{u}  \ dx ds\\
        &= \int^t_0 \int_{\mathbb{T}^3_x}\frac{1}{2} \mu \sigma(\tilde{u}-u^b):\sigma(\tilde{u}-u^b) - \frac{1}{2} \mu \sigma(u^b):\sigma(u^b) \ dxds + R_5,
    \end{aligned}
\end{equation}
where
\begin{equation}
    \begin{aligned}
        R_5&= \int^t_0 \int_{\mathbb{T}^3_x}  \nabla_x \mdot [ \mu(\rho,\theta) \sigma(\tilde{u})] \mdot (\frac{-\varepsilon\tilde{\rho} -\varepsilon \tilde{\theta} -\varepsilon^2 \tilde{\rho}\tilde{\theta}}{\rho \theta})(u^b-\tilde{u})+ (\frac{\varepsilon\theta}{\theta})\mu \sigma(u^b):\nabla_x \tilde{u} \ dxds\\
           &+\int^t_0 \int_{\mathbb{T}^3_x} \frac{1}{2} \left( \mu(\rho,\theta) - \mu \right)\sigma(\tilde{u}):\sigma(\tilde{u}-u^b) \ dxds.
    \end{aligned}
\end{equation}
For the term II, same reasoning shows that
\begin{equation}\label{integration_by_part_theta}
    \begin{aligned}
        \int^t_0 \int_{\mathbb{T}^3_x} \frac{1}{\rho}\left(\frac{5}{2}  \nabla_x \mdot [\kappa(\rho,\theta) \nabla_x \tilde{\theta}]\right) \frac{1}{\theta^2} (\theta^b -\tilde{\theta}) -\frac{1}{\theta^2}\frac{5}{2}\kappa \nabla_x \theta^b \cdot \nabla_x \tilde{\theta} \ dx ds \\
        = \int^t_0 \int_{\mathbb{T}^3_x} \frac{5}{2} \kappa(\nabla_x \tilde{\theta} -\nabla_x \theta^b)^2 -\frac{5}{2}\kappa( \nabla_x \theta^b )^2 \ dxds + R_6,
    \end{aligned}
\end{equation}
where
\begin{equation}
    \begin{aligned}
        R_6&= \int^t_0 \int_{\mathbb{T}^3_x} (\frac{1}{\rho \theta^2}-1 ) \frac{5}{2} \nabla_x \mdot [\kappa(\rho,\theta) \nabla_x \tilde{\theta} ](\theta^b -\tilde{\theta}) + (1-\frac{1}{\theta^2})\kappa \nabla_x \theta^b \cdot \nabla_x \tilde{\theta}\ dxds\\
         &+ \int^t_0 \int_{\mathbb{T}^3_x} \frac{5}{2} (\kappa - \kappa(\rho,\theta))\nabla_x \tilde{\theta}\cdot (\nabla_x \tilde{\theta}-\nabla_x \theta^b) \ dxds.
    \end{aligned}
\end{equation}

\subsection{Relative entropy control}\label{SS_REC}
The terms \Rmnum{3} and \Rmnum{4} are unsigned, however they can be controlled by the quadratic of $(\rho^b-\tilde{\rho}^{\varepsilon}, u^b - \tilde{u}^{\varepsilon},\theta^b - \tilde{\theta}^{\varepsilon})$ which  can be controlled by relative entropy $\frac{1}{\varepsilon^2} H(f_{\varepsilon}|M_{\varepsilon})$.
In fact, we have the following two lemmas.
\begin{lemma} \label{same_moments}
    \begin{equation}
        \begin{aligned}
            H(f_{\varepsilon}|M_{\varepsilon}) = H(f_{\varepsilon}| M_{f_\varepsilon}) + H(M_{f_{\varepsilon}}|M_{\varepsilon}),
        \end{aligned}
    \end{equation}
    where $M_{f_{\varepsilon}}=M(\rho^f,u^f,\theta^f)$ is a local Maxwellian sharing the same moments with $f_{\varepsilon}$, i.e.,
    \begin{equation}
        \begin{aligned}
            \int_{\mathbb{R}^3_v} f_{\varepsilon} -M_{f_{\varepsilon}} dv =0, \\
            \int_{\mathbb{R}^3_v} v\left(f_{\varepsilon} -M_{f_{\varepsilon}}\right) dv = 0,\\
            \int_{\mathbb{R}^3_v} |v|^2\left(f_{\varepsilon} -M_{f_{\varepsilon}}\right) dv = 0.
        \end{aligned}
    \end{equation}
\end{lemma}
\begin{lemma}\label{difference_moments_control}
    \begin{equation}
        \begin{aligned}
            \frac{1}{\varepsilon^2}H(M_{f_{\varepsilon}}|M_{\varepsilon}) = \frac{1}{2}\left( \left( \rho^b -\tilde{\rho}^{\varepsilon}\right)^2+ \frac{3}{2}\left( \theta^b -\tilde{\theta}^{\varepsilon}\right)^2 + \left( u^b -\tilde{u}^{\varepsilon}\right)^2 \right) + R_8 + R_9+ R_{10},
        \end{aligned}
    \end{equation}
    where $R_8$, $R_9$ and $R_{10}$ are expected to tend to zero as $\varepsilon$ tend to zero. The explicit form of $R_8$, $R_9$ and $R_{10}$ will be given in the proof of this lemma.
\end{lemma}
Note that the relative entropy is always positive, Lemma \ref{same_moments} and \ref{difference_moments_control} then implies
\begin{equation}\label{convection_control}
    \begin{aligned}
        \int^t_0 \int_{\mathbb{T}^3_x} \frac{1}{\theta} & \left[ (\tilde{u}-u^b)\otimes(\tilde{u}-u^b)  - \frac{1}{3} (\tilde{u} - u^b)^2 I   \right] : \nabla_x \tilde{u} +
        \left[\frac{5}{2} \frac{1}{\theta^2} (\tilde{u}-u^b)(\tilde{\theta}-\theta^b) \right] \cdot \nabla_x \tilde{\theta}  \ dx ds   \\
        &\leqslant C(\| \theta \|_{L^{\infty}(dtdx)} ) \int^t_0 \int_{\mathbb{T}^3_x} \| (\nabla_x \tilde{u},\nabla_x \tilde{\theta})\|_{L^{\infty}(dx)} \left[\left(\tilde{u}-u^b\right)^2 + \left(\tilde{\theta} - \theta^b \right)^2\right] \ dxds\\
        &\leqslant C(\| \theta \|_{L^{\infty}(dtdx)} ) \int^t_0   \| (\nabla_x \tilde{u},\nabla_x \tilde{\theta})\|_{L^{\infty}(dx)} \frac{1}{\varepsilon^2} H(f_{\varepsilon}|M_{\varepsilon}) \ ds +R_7,
    \end{aligned}
\end{equation}
where
\begin{equation}
    \begin{aligned}
        R_7=C(\| \theta \|_{L^{\infty}(dtdx)} )\varepsilon \int^t_0 \int_{\mathbb{T}^3_x}    \| (\nabla_x \tilde{u},\nabla_x \tilde{\theta})\|_{L^{\infty}(dx)} ( R_8 + R_9 +R_{10}) \ dx ds,
    \end{aligned}
\end{equation}
and $C(\| \theta \|_{L^{\infty}(dtdx)} )$ is a constant depends only on $\| \theta \|_{L^{\infty}(dtdx)}$.

Combing now \eqref{temp_relative_entropy},  \eqref{integration_by_part_u}, \eqref{integration_by_part_theta} and \eqref{convection_control} leads to
\begin{equation} \label{mid_result}
    \begin{aligned}
        \frac{1}{{\varepsilon}^2} &H(f_{\varepsilon}| M_{\varepsilon})(t) +\int_0^t \int_{\mathbb{T}^3_x}\frac{1}{\varepsilon^4}D(f_\varepsilon)(s,x) - \frac{1}{2} \mu \sigma(u^b): \sigma(u^b) - \frac{5}{2}\kappa (\nabla_x \theta^b)^2\  dxds \\
        &+\int_0^t \int_{\mathbb{T}^3_x} \frac{1}{2}\mu \sigma(\tilde{u}-u^b):\sigma(\tilde{u}-u^b) + \frac{5}{2}  \kappa(\nabla_x \tilde{\theta} -\nabla_x \theta^b)^2 \ dxds\\
        &\lesssim \frac{1}{\varepsilon^2} H(f^{\mathrm{in}}_{\varepsilon}| M^{\mathrm{in}}_{\varepsilon}) + \int^t_0 \|(\nabla_x \tilde{u},\nabla_x \tilde{\theta})\|_{L^{\infty}(dx)} \frac{1}{\varepsilon^2} H(f_{\varepsilon}|M_{\varepsilon}) ds +\tilde{R} + R_5+R_6+R_7.
    \end{aligned}
\end{equation}

\subsection{Entropy dissipation control}\label{SS_EDC}
 We now deal with the term
\begin{equation}\label{dissipation_control}
    \begin{aligned}
        \int_0^t \int_{\mathbb{T}^3_x}\frac{1}{\varepsilon^4}D(f_\varepsilon)(s,x) - \frac{1}{2} \mu \sigma(u^b): \sigma(u^b) - \frac{5}{2}\kappa (\nabla_x \theta^b)^2\  dxds ,
    \end{aligned}
\end{equation}
which means that we need to use the entropy dissipation to control the viscosity and heat-conductivity determined by the fluid part of the Boltzmann equation. More specifically, we want to show that
this term can be decomposed as some positive terms and some other terms that tend to zero. In fact, we shall prove the following inequality

\begin{equation} \label{entropy_dissipation_control}
    \begin{aligned}
        \int_0^t &\int_{\mathbb{T}^3_x}\frac{1}{\varepsilon^4}D(f_\varepsilon)(s,x) - \frac{1}{2} \mu \sigma(u^b): \sigma(u^b) - \frac{5}{2}\kappa (\nabla_x \theta^b)^2 \  dxds -\mathcal{O}(\varepsilon) \geqslant 0.
    \end{aligned}
\end{equation}
For convenience of notation, we define

\begin{equation}
    \begin{aligned}
        q_{\varepsilon} := \frac{G^{\prime}_{\varepsilon}G^{\prime}_{\varepsilon 1}-G_{\varepsilon} G_{\varepsilon 1}}{\varepsilon^2}.
    \end{aligned}
\end{equation}
And we have the following lemma
\begin{lemma}\label{D_q_equivlence}
    \begin{equation}
        \int^t_0 \int_{\mathbb{T}^3_x} \frac{1}{\varepsilon^4} D(f_{\varepsilon})\ dxds = \frac{1}{4}  \int^t_0 \int_{\mathbb{T}^3_x} \langle \langle q_{\varepsilon}^2  \rangle \rangle \ dx ds + R_{11},
    \end{equation}
    where
    \begin{equation}
        \begin{aligned}
            R_{11}= \int^t_0 \int_{\mathbb{T}^3_x}
             \langle \langle-\frac{1}{2} q_{\varepsilon}^2 \left(\varepsilon g^{\prime}_{\varepsilon 1} + \varepsilon g^{\prime}_{\varepsilon} + \varepsilon^2 g^{\prime}_{\varepsilon } g^{\prime}_{\varepsilon 1} + \varepsilon g^{\prime}_{\varepsilon 1} + \varepsilon g^{\prime}_{\varepsilon} + \varepsilon^2 g^{\prime}_{\varepsilon } g^{\prime}_{\varepsilon 1} \right) + \frac{1}{\varepsilon^2}r_3 q_{\varepsilon} - \frac{1}{\varepsilon^2}r_4 q_{\varepsilon} \rangle \rangle \ dxds
        \end{aligned}
    \end{equation}
\end{lemma}

On the other hand, we have the following lemma
\begin{lemma}\cite{BGL2}
    Any $q$ satisfies the inequality
    \begin{equation}
        \begin{aligned}
            \frac{1}{2}\frac{1}{\mu}\langle\!\langle  \hat{A} q \rangle\!\rangle :\langle\!\langle  \hat{A} q \rangle\!\rangle + \frac{2}{5}\frac{1}{\kappa}\langle\!\langle  \hat{B} q \rangle\!\rangle \cdot \langle\!\langle  \hat{B} q \rangle\!\rangle
            \leqslant \frac{1}{4} \langle\!\langle   q^2 \rangle\!\rangle,
        \end{aligned}
    \end{equation}
    if $\langle\!\langle   q^2 \rangle\!\rangle$ is bounded.
\end{lemma}
The term \eqref{dissipation_control} can therefore be rewritten as
\begin{equation} \label{dissipation_control_mid_result}
    \begin{aligned}
        \int_0^t &\int_{\mathbb{T}^3_x}\frac{1}{\varepsilon^4}D(f_\varepsilon)(s,x) - \frac{1}{2} \mu \sigma(u^b): \sigma(u^b) - \frac{5}{2}\kappa (\nabla_x \theta^b)^2\  dxds \\
        &=\int_0^t \int_{\mathbb{T}^3_x}\frac{1}{\varepsilon^4}D(f_\varepsilon)(s,x) - \frac{1}{4}\langle\!\langle {q}_{\varepsilon}^2 \rangle \!\rangle \  dx ds\\
        &+\int_0^t \int_{\mathbb{T}^3_x} \frac{1}{4}\langle\!\langle {q}_{\varepsilon}^2 \rangle \!\rangle - \frac{1}{2}\frac{1}{\mu}\langle\!\langle  \hat{A} {q}_{\varepsilon} \rangle\!\rangle :\langle\!\langle  \hat{A} {q}_{\varepsilon} \rangle\!\rangle - \frac{2}{5}\frac{1}{\kappa}\langle\!\langle  \hat{B} {q}_{\varepsilon} \rangle\!\rangle \cdot \langle\!\langle  \hat{B} {q}_{\varepsilon} \rangle\!\rangle \ dxds\\
        &+\int_0^t \int_{\mathbb{T}^3_x} \frac{1}{2}\frac{1}{\mu} \left(\langle\!\langle  \hat{A} {q}_{\varepsilon} \rangle\!\rangle - \mu \sigma(u^b)  \right) :\left(\langle\!\langle  \hat{A} {q}_{\varepsilon} \rangle\!\rangle -\mu \sigma(u^b)  \right)
        +\left(\langle\!\langle  \hat{A} {q}_{\varepsilon} \rangle\!\rangle - \mu \sigma(u^b)  \right) : \sigma(u^b) \ dxds \\
        &+\int_0^t \int_{\mathbb{T}^3_x} \frac{2}{5}\frac{1}{\kappa} \left( \langle\!\langle  \hat{B} {q}_{\varepsilon} \rangle\!\rangle - \frac{5}{2}\kappa \sigma(u^b)\right)\cdot\left( \langle\!\langle  \hat{B} {q}_{\varepsilon} \rangle\!\rangle - \frac{5}{2}\kappa \sigma(u^b)\right)
        + 2 \left( \langle\!\langle  \hat{B} {q}_{\varepsilon} \rangle\!\rangle - \frac{5}{2}\kappa \sigma(u^b)\right) \cdot \nabla_x \theta^b \ dxds.
    \end{aligned}
\end{equation}
Aside from all the positive terms, we consider
\begin{equation}
    \begin{aligned}
        \int_0^t \int_{\mathbb{T}^3_x} \left(\langle\!\langle  \hat{A} q_{\varepsilon} \rangle\!\rangle - \mu \sigma(u^b)  \right) : \sigma(u^b) \ dxds,
    \end{aligned}
\end{equation}
and
\begin{equation}
    \begin{aligned}
        \int_0^t \int_{\mathbb{T}^3_x} \left( \langle\!\langle  \hat{B} q_{\varepsilon} \rangle\!\rangle - \frac{5}{2}\kappa \sigma(u^b)\right) \cdot \nabla_x \theta^b \ dxds.
    \end{aligned}
\end{equation}
\begin{lemma}\label{q_asymptotics}
    \begin{equation} \label{A_q}
        \begin{aligned}
            \int_0^t& \int_{\mathbb{T}^3_x} \left(\langle\!\langle  \hat{A} q_{\varepsilon} \rangle\!\rangle - \mu \sigma(u^b)  \right) : \sigma(u^b) \ dxds \\
            &= \int_0^t \int_{\mathbb{T}^3_x} \left( \langle v \mdot \nabla_x \Loth g_{\varepsilon} , \hat{A}\rangle + \langle \varepsilon \partial_t g_{\varepsilon} , \hat{A}\rangle\right) :\sigma(u^b) \ dxds ,\\
            \end{aligned}\\
    \end{equation}
    \begin{equation}\label{B_q}
        \begin{aligned}
            \int_0^t & \int_{\mathbb{T}^3_x} \left( \langle\!\langle  \hat{B} q_{\varepsilon} \rangle\!\rangle - \frac{5}{2}\kappa \sigma(u^b)\right) \cdot \nabla_x \theta^b \ dxds \\
            &=\int_0^t \int_{\mathbb{T}^3_x}  \left(\langle v \mdot \nabla_x \Loth g_{\varepsilon} , \hat{B}\rangle + \langle \varepsilon \partial_t g_{\varepsilon} , \hat{B}\rangle\right) \cdot \nabla_x \theta^b \ dxds.
            \end{aligned}
    \end{equation}
\end{lemma}

Combing now \eqref{dissipation_control_mid_result}, \eqref{A_q} and \eqref{B_q} yields that
\begin{equation} \label{dissipation_control_final_result}
    \begin{aligned}
        \int_0^t &\int_{\mathbb{T}^3_x}\frac{1}{\varepsilon^4}D(f_\varepsilon)(s,x) - \frac{1}{2} \mu \sigma(u^b): \sigma(u^b) - \frac{5}{2}\kappa (\nabla_x \theta^b)^2 \  dxds -R_{11} - R_{12} - 2R_{13} \geqslant 0,
    \end{aligned}
\end{equation}
where
\begin{equation}
    \begin{aligned}
        R_{12}= \int_0^t \int_{\mathbb{T}^3_x} \left( \langle v \mdot \nabla_x \Loth g_{\varepsilon} , \hat{A}\rangle + \langle \varepsilon \partial_t g_{\varepsilon} , \hat{A}\rangle\right) :\sigma(u^b) \ dxds ,\\
        R_{13}=\int_0^t \int_{\mathbb{T}^3_x}  \left(\langle v \mdot \nabla_x \Loth g_{\varepsilon} , \hat{B}\rangle + \langle \varepsilon \partial_t g_{\varepsilon} , \hat{B}\rangle\right) \cdot \nabla_x \theta^b \ dxds.
    \end{aligned}
\end{equation}

\subsection{Conclusion} \label{SS_Conclusion}
We now complete the proof of Theorem \ref{main}.
Note that
\begin{equation}\label{quadratic_term}
    \int_0^t \int_{\mathbb{T}^3_x} \frac{1}{2}\mu \sigma(\tilde{u}-u^b):\sigma(\tilde{u}-u^b) + \frac{5}{2}  \kappa(\nabla_x \tilde{\theta} -\nabla_x \theta^b)^2 \ dxds \geqslant 0.
\end{equation}
Combining with \eqref{mid_result} and \eqref{dissipation_control_final_result} we have
\begin{equation}
    \begin{aligned}
        \frac{1}{{\varepsilon}^2} &H(f_{\varepsilon}| M_{\varepsilon})(t)  \\
        &\lesssim \frac{1}{\varepsilon^2} H(f^{\mathrm{in}}_{\varepsilon}| M^{\mathrm{in}}_{\varepsilon}) + \int^t_0 \|(\nabla_x \tilde{u},\nabla_x \tilde{\theta})\|_{L^{\infty}(dx)} \frac{1}{\varepsilon^2} H(f_{\varepsilon}|M_{\varepsilon}) ds \\
        & +\tilde{R}+ R_5+R_6+R_7 -R_{11}-R_{12} - 2R_{13}.
    \end{aligned}
\end{equation}
All the remainder terms can be shown to tend to zero as $\varepsilon$ tends to zero by using that $(\rho^b_{\varepsilon}, u^b_{\varepsilon}, \theta^b_{\varepsilon})$, $(\tilde{\rho}_{\varepsilon},\tilde{u}_{\varepsilon},\tilde{\theta}_{\varepsilon})$, $g_{\varepsilon}$ and $\frac{1}{\varepsilon} \Loth g_{\varepsilon}$
are separately uniformly bounded in some proper functional spaces, and Grönwall's inequality then implies that
\begin{equation}
    \begin{aligned}
       \frac{1}{{\varepsilon}^2} H(f_{\varepsilon}| M_{\varepsilon})(t)  \leqslant
       \left( \frac{1}{\varepsilon^2} H(f^{\mathrm{in}}_{\varepsilon}| M^{\mathrm{in}}_{\varepsilon})+  \mathcal{O}(\varepsilon)\right) \exp \left( C \int_0^{\infty}  \| (\nabla_x  \tilde{u}, \nabla_x \tilde{\theta}) \|_{L^{\infty}(dx)} dt \right).
    \end{aligned}
\end{equation}
Hence,
\begin{equation}\label{entropy_control}
    \begin{aligned}
       \frac{1}{{\varepsilon}^2} H(f_{\varepsilon}| M_{\varepsilon})(t)  \lesssim
       \frac{1}{\varepsilon^2} H(f^{\mathrm{in}}_{\varepsilon}| M^{\mathrm{in}}_{\varepsilon})+  \mathcal{O}(\varepsilon).
    \end{aligned}
\end{equation}
This estimate together with \eqref{mid_result}, \eqref{quadratic_term} and the fact $H(f|g) \geqslant 0$ leads to
\begin{equation}\label{entropy_dissipation_bound_temp}
    \int_0^t \int_{\mathbb{T}^3_x}\frac{1}{\varepsilon^4}D(f_\varepsilon)(s,x) - \frac{1}{2} \mu \sigma(u^b): \sigma(u^b) - \frac{5}{2}\kappa (\nabla_x \theta^b)^2\  dxds \lesssim \frac{1}{\varepsilon^2} H(f_{\varepsilon}^{\mathrm{in}}|M_{\varepsilon}^{\mathrm{in}}) + \mathcal{O}(\varepsilon).
  \end{equation}
 Using \eqref{mid_result},   \eqref{entropy_control}, \eqref{entropy_dissipation_bound_temp} and \eqref{dissipation_control_final_result},we then have
 \begin{equation}
    \int_0^t \int_{\mathbb{T}^3_x} \frac{1}{2}\mu \sigma(\tilde{u}-u^b):\sigma(\tilde{u}-u^b) + \frac{5}{2}  \kappa(\nabla_x \tilde{\theta} -\nabla_x \theta^b)^2 \ dxds \lesssim \frac{1}{\varepsilon^2} H(f_{\varepsilon}^{\mathrm{in}}|M_{\varepsilon}^{\mathrm{in}}) + \mathcal{O}(\varepsilon),
\end{equation}
which completes the proof of Theorem \ref{main}.

\section{Appendix}
  We present the proof of the lemmas we used in this paper.
    \begin{proof}[\textbf{Proof of Lemma \ref{observation_lemma}}]
        Using the explicit form of $\mathcal{M}(\rho,u,\theta)$ and the relation $V := \frac{v - u}{\sqrt{\theta}}$, we have
       \begin{equation}\label{observation_temp1}
        \begin{aligned}
            v \cdot &\nabla_x \log \mathcal{M}(\rho,u,\theta) = v \mdot \nabla_x  \left(\log \rho   -\frac{3}{2} \log \theta - \frac{|u-v|^2}{2\theta} \right)\\
            &= \frac{v \mdot \nabla_x \rho }{\rho} - \frac{3}{2} \frac{v \mdot \nabla_x \theta}{\theta}  + \frac{|v-u|^2}{2\theta^2} v \mdot \nabla_x \theta - \frac{1}{2\theta} (v \mdot \nabla_x u) \cdot 2(u-v)\\
            &= \frac{(\sqrt{\theta}V + u)\cdot \nabla_x \rho}{\rho} - \frac{3}{2} \frac{(\sqrt{\theta}V + u)\cdot \nabla_x \theta}{\theta} + \frac{|V|^2}{2}\frac{(\sqrt{\theta}V + u)\cdot \nabla_x \theta}{\theta}  + \frac{(\sqrt{\theta}V + u )\cdot \nabla_x u}{\sqrt{\theta}} \cdot V\\
            &= \frac{1}{\rho} \left( u \cdot \nabla_x \rho \right) + (\frac{\theta\nabla_x\rho}{\rho}-\frac{3}{2} \nabla_x \theta  + u \mdot \nabla_x u) \mdot \frac{V}{\sqrt{\theta}}  + (u \mdot \nabla_x \theta) \frac{1}{\theta}(\frac{|V|^2}{2}-\frac{3}{2}) \\
            &+ \frac{|V|^2}{2}V \cdot \frac{\nabla_x \theta}{\sqrt{\theta}} + V \otimes V : \nabla_x u.
        \end{aligned}
       \end{equation}
       Note that
    \begin{equation}\label{observation_temp2}
        \begin{aligned}
            \frac{|V|^2}{2}V  \mdot \frac{\nabla_x \theta}{\sqrt{\theta}}  &= (B(V)  + \frac{5}{2}V ) \mdot \frac{\nabla_x \theta}{\sqrt{\theta}} = B(V) \mdot \frac{\nabla_x \theta}{\sqrt{\theta}} + \frac{5}{2}\nabla_x \theta \mdot \frac{V}{\sqrt{\theta}},\\
            V\otimes V : \nabla_x u &=( A(V) + \frac{2}{3} (\frac{|V|^2}{2}-\frac{3}{2})I + I): \nabla_x u \\
            &=  A(V): \nabla_x u +\frac{2}{3} \theta \nabla_x \mdot u (\frac{1}{\theta} (\frac{|V|^2}{2}-\frac{3}{2})) + \nabla_x \mdot u.
        \end{aligned}
    \end{equation}
    Combing \eqref{observation_temp1} and \eqref{observation_temp2} leads to
    \begin{equation}
        \begin{aligned}
            v \cdot &\nabla_x \log \mathcal{M}(\rho,u,\theta) = (u \mdot \nabla_x \rho + \rho \nabla_x \mdot u)\frac{1}{\rho} + ( u \mdot \nabla_x u  + \frac{\theta}{\rho} \nabla_x \rho + \nabla_x \theta)\frac{V}{\sqrt{\theta}}\\
            &+ (u \mdot \nabla_x \theta
            + \frac{2}{3} \theta  \nabla_x \mdot u) \frac{1}{\theta}(\frac{|V|^2}{2}-\frac{3}{2})+A(V):\nabla_x u + B(V)\cdot \frac{\nabla_x \theta}{\sqrt{\theta}},
        \end{aligned}
    \end{equation}
    which completes the proof.

    \end{proof}

  \begin{proof}[\textbf{Proof of Lemma \ref{AvBv_asymptotics}}]
   To prove this lemma, we need to use the following two results which we refer to \cites{BGL1,BGL2} for their proof.
   \begin{lemma}\cite{BGL1} \label{Qgg_Lg}
    If $g \in Ker \mathcal{L}$, then we have
    \begin{equation}
        Q(g,g) = \mathcal{L} g^2.
    \end{equation}
\end{lemma}
\begin{lemma}\cite{BGL2} \label{Tensor_Product}
    \begin{equation}
        \begin{aligned}
        \langle \hat{A} \otimes A \rangle &= \mu (\delta_{ik} \delta_{jl} + \delta_{il} \delta_{jk} -\frac{2}{3} \delta_{ij}\delta_{kl}), \\
        \langle \hat{B} \otimes B \rangle &= \frac{5}{2} \kappa \delta_{ij}.
        \end{aligned}
    \end{equation}
\end{lemma}
    Using \eqref{perturbed_Boltzmann} we have
\begin{equation}
    \langle A(v), g_{\varepsilon} \rangle = \langle \hat{A}(v),\mathcal{L} g_{\varepsilon} \rangle = \langle \hat{A}(v), -\varepsilon^2 \partial_t g_{\varepsilon} - \varepsilon v\mdot \nabla_x g_{\varepsilon} + \varepsilon Q(g_{\varepsilon},g_{\varepsilon}) \rangle.
\end{equation}
By Lemma \ref{Tensor_Product} and \eqref{v_dot_nabla_g}, we readily have
\begin{equation}
    \begin{aligned}
        \langle \hat{A}(v),v\mdot \nabla_x g_{\varepsilon} \rangle &= \langle \hat{A}(v),v\mdot \nabla_x \Lpre g_{\varepsilon} \rangle + \langle \hat{A}(v),v\mdot \nabla_x \Loth g_{\varepsilon} \rangle\\
        &= \langle \hat{A}(v)\otimes A(v) \rangle : \nabla_x u^b + \langle \hat{A}(v),v\mdot \nabla_x \Loth g_{\varepsilon} \rangle\\
        &=\mu \sigma(u^b) + \langle \hat{A}(v),v\mdot \nabla_x \Loth g_{\varepsilon} \rangle.
    \end{aligned}
\end{equation}
Using Lemma \ref{Qgg_Lg} results in
\begin{equation}
    \begin{aligned}
        \langle \hat{A}(v), Q(g_{\varepsilon},g_{\varepsilon})  \rangle &=\langle \hat{A}(v), Q(\Lpre g_{\varepsilon},\Lpre g_{\varepsilon})  \rangle + 2\langle \hat{A}(v), Q(\Loth g_{\varepsilon},\Lpre g_{\varepsilon})  \rangle + \langle \hat{A}(v), Q(\Loth g_{\varepsilon},\Loth g_{\varepsilon})  \rangle\\
        &=\langle \hat{A}(v),\frac{1}{2} \mathcal{L}(\Lpre g_{\varepsilon}^2)  \rangle +2\langle \hat{A}(v), Q(\Loth g_{\varepsilon},\Lpre g_{\varepsilon})  \rangle + \langle \hat{A}(v), Q(\Loth g_{\varepsilon},\Loth g_{\varepsilon})  \rangle\\
        &=\frac{1}{2}\langle {A}(v) \otimes A(v) \rangle : u^b\otimes u^b+2\langle \hat{A}(v), Q(\Loth g_{\varepsilon},\Lpre g_{\varepsilon})  \rangle + \langle \hat{A}(v), Q(\Loth g_{\varepsilon},\Loth g_{\varepsilon})  \rangle\\
        &= u^b \otimes u^b - \frac{(u^b)^2}{3}I +2\langle \hat{A}(v), Q(\Loth g_{\varepsilon},\Lpre g_{\varepsilon})  \rangle + \langle \hat{A}(v), Q(\Loth g_{\varepsilon},\Loth g_{\varepsilon})  \rangle.
    \end{aligned}
\end{equation}
And similarly,
\begin{equation}
    \begin{aligned}
        \langle \hat{B}(v),v\mdot \nabla_x g_{\varepsilon} \rangle
        &=\frac{5}{2}\kappa \nabla_x \theta^b + \langle \hat{B}(v),v\mdot \nabla_x \Loth g_{\varepsilon} \rangle.
    \end{aligned}
\end{equation}
\begin{equation}
    \begin{aligned}
        \langle \hat{B}(v), Q(g_{\varepsilon},g_{\varepsilon})  \rangle
        &=\frac{5}{2}u^b \theta^b+2\langle \hat{B}(v), Q(\Loth g_{\varepsilon},\Lpre g_{\varepsilon})  \rangle + \langle \hat{B}(v), Q(\Loth g_{\varepsilon},\Loth g_{\varepsilon})  \rangle.
    \end{aligned}
\end{equation}
In summary, we have
\begin{equation}
    \begin{aligned}
    \frac{1}{\varepsilon}\langle A(v), g_{\varepsilon} \rangle &=-\mu \sigma(u^b) +\left(u^b \otimes u^b - \frac{(u^b)^2}{3}I\right) + R_A,\\
    \frac{1}{\varepsilon}\langle B(v), g_{\varepsilon} \rangle &=-\frac{5}{2}\kappa \nabla_x \theta^b +\frac{5}{2}u^b \theta^b  + R_B.
    \end{aligned}
\end{equation}
\end{proof}

\begin{proof}[\textbf{Proof of Lemma \ref{same_moments}}]
    Since $f_{\varepsilon}$ and $M_{f_{\varepsilon}}$ have the same moments, we have
    \begin{equation}
        \begin{aligned}
            \int_{\mathbb{R}^3_v} f_{\varepsilon} \log M_{f_{\varepsilon}} -  M_{f_{\varepsilon}} \log  M_{f_{\varepsilon}} dv =0,\\
            \int_{\mathbb{R}^3_v} f_{\varepsilon} \log M_{{\varepsilon}} -  M_{f_{\varepsilon}} \log  M_{{\varepsilon}} dv =0.
        \end{aligned}
    \end{equation}
    Hence,
    \begin{equation}
        \begin{aligned}
            H(f_{\varepsilon}| M_{\varepsilon})&= \int_{\mathbb{T}^3_x}  \int_{\mathbb{R}^3_v}  f_{\varepsilon} \log f_{\varepsilon} -f_{\varepsilon} \log M_{\varepsilon} -f_{\varepsilon} + M_{\varepsilon} \ dv dx\\
            &= \int_{\mathbb{T}^3_x}  \int_{\mathbb{R}^3_v} f_{\varepsilon} \log f_{\varepsilon} - f_{\varepsilon} \log M_{f_{\varepsilon}} +   f_{\varepsilon} \log M_{f_{\varepsilon}} - f_{\varepsilon}\log M_{\varepsilon} -f_{\varepsilon} + M_{\varepsilon} \ dv dx \\
            &= \int_{\mathbb{T}^3_x}  \int_{\mathbb{R}^3_v} f_{\varepsilon} \log f_{\varepsilon} - {f_{\varepsilon}} \log M_{f_{\varepsilon}}  -f_{\varepsilon } + M_{f_{\varepsilon}} \ dv dx\\
            &+ \int_{\mathbb{T}^3_x}  \int_{\mathbb{R}^3_v} M_{f_{\varepsilon}} \log M_{f_{\varepsilon}} - M_{f_{\varepsilon}}\log M_{\varepsilon} - M_{f_{\varepsilon} } + M_{\varepsilon} \ dv dx\\
            &= H(f_{\varepsilon}| M_{f_{\varepsilon}}) + H(M_{f_{\varepsilon}}| M_{\varepsilon}).
        \end{aligned}
    \end{equation}
\end{proof}
\begin{proof}[\textbf{Proof of Lemma \ref{difference_moments_control}}]
    Note that
    \begin{equation}
        M_{f_{\varepsilon}} = \mathcal{M} (\rho_{f_{\varepsilon}},u_{f_{\varepsilon}},\theta_{f_{\varepsilon}} ),
    \end{equation}
    by \eqref{g_moments}, we have
    \begin{equation}
        \begin{aligned}
            &\rho_{f_{\varepsilon}} = \int_{\mathbb{R}^3_v} f_{\varepsilon} \ dv  = 1+ \varepsilon \langle g_{\varepsilon} \rangle   = 1+ \varepsilon \rho^b,\\
            &\rho_{f_{\varepsilon}}  u_{f_{\varepsilon}} =\int_{\mathbb{R}^3_v} v f_{\varepsilon}\ dv =  \varepsilon \langle v g_{\varepsilon} \rangle = \varepsilon u^b ,\\
            &\rho_{f_{\varepsilon}} \left(\theta_{f_{\varepsilon}}  + \frac{|u_{f_{\varepsilon}}|^2}{3} -1\right) =\int_{\mathbb{R}^3_v} \frac{|v|^2-3}{3} f_{\varepsilon} \ dv  =  \varepsilon \langle \frac{|v|^2-3}{3} g_{\varepsilon} \rangle = \varepsilon \theta^b. \\
        \end{aligned}
    \end{equation}
    Simple calculation then shows that
    \begin{equation}
        \begin{aligned}
            \rho_{f_{\varepsilon}} &= 1 + \varepsilon \rho^b, \quad \quad
            u_{f_{\varepsilon}} =  \varepsilon u^b - \varepsilon^2 \frac{\rho^b u^b}{1+ \varepsilon \rho^b} =\varepsilon u^b+ \varepsilon^2 r_1,\\
            \theta_{f_{\varepsilon}} &= 1 + \varepsilon \theta^b -\varepsilon^2 \left(\frac{\theta^b \rho^b }{1 + \varepsilon \rho^b} + \frac{1}{3}(u^b)^2 -\frac{2\varepsilon(u^b)^2 \rho^b}{3(1+\varepsilon \rho^b)} +  \varepsilon^2  \frac{(\rho^b u^b)^2}{3(1+\varepsilon \rho^b)^2}  \right)
            =1 + \varepsilon \theta^b + \varepsilon^2 r_2.
        \end{aligned}
    \end{equation}
    Now by definition of relative entropy, we have
    \begin{equation}
        \begin{aligned}
            \frac{1}{\varepsilon^2 } H(M_{f_{\varepsilon}}&|M_{\varepsilon}) = \frac{1}{\varepsilon^2 } \int_{\mathbb{T}^3_x}  \int_{\mathbb{R}^3_v} M_{f_{\varepsilon}} \log M_{f_{\varepsilon}} - M_{f_{\varepsilon}} \log M_{\varepsilon}\ dv dx + \frac{1}{\varepsilon^2 } \int_{\mathbb{T}^3_x} - \rho_{f_{\varepsilon}} + 1+ \varepsilon \tilde{\rho} \ dx  \\
            &=\frac{1}{\varepsilon^2 } \int_{\mathbb{T}^3_x}  \int_{\mathbb{R}^3_v} M_{f_{\varepsilon}} \left( \log \rho_{f_{\varepsilon}} - \log(1+ \varepsilon \tilde{\rho}) -\frac{3}{2} \log\theta_{f_{\varepsilon}} + \frac{3}{2} \log(1+\varepsilon \tilde{\theta}) \right) \ dv dx \\
            &-\frac{1}{\varepsilon^2 } \int_{\mathbb{T}^3_x}  \int_{\mathbb{R}^3_v} M_{f_{\varepsilon}} \left( \frac{|u_{f_{\varepsilon}}-v|^2}{2\theta_{f_{\varepsilon}}} + \frac{|\varepsilon \tilde{u} - v|^2}{2(1+\varepsilon\theta)}\right) \ dv dx + \frac{1}{\varepsilon^2 } \int_{\mathbb{T}^3_x} - \varepsilon \rho^b + \varepsilon \tilde{\rho} \ dx\\
            &= \frac{1}{\varepsilon^2 } \int_{\mathbb{T}^3_x} \rho_{f_{\varepsilon}} \left(\log \rho_{f_{\varepsilon}} - \log(1+ \varepsilon \tilde{\rho}) -\frac{3}{2} \log\theta_{f_{\varepsilon}} + \frac{3}{2} \log(1+\varepsilon \tilde{\theta})\right)\ \ dx \\
            &+ \frac{1}{\varepsilon^2 } \int_{\mathbb{T}^3_x} \rho_{f_{\varepsilon}} \left(\frac{3}{2(1+\varepsilon \tilde{\theta})} (\theta_{f_{\varepsilon}}-1-\varepsilon \tilde{\theta})  + \frac{|u_{f_{\varepsilon}}- \varepsilon \tilde{u}|^2}{2(1+\varepsilon \tilde{\theta})}\right) - \varepsilon \rho^b + \varepsilon \tilde{\rho} \ dx.\\
        \end{aligned}
    \end{equation}
    Note that
    \begin{equation}\label{R8}
        \begin{aligned}
            \frac{1}{\varepsilon^2 }& \int_{\mathbb{T}^3_x} \rho_{f_{\varepsilon}} \left(\log \rho_{f_{\varepsilon}} - \log(1+ \varepsilon \tilde{\rho}) \right) - \varepsilon \rho^b + \varepsilon \tilde{\rho} \ dx\\
            &= \frac{1}{\varepsilon^2 } \int_{\mathbb{T}^3_x} (1 + \varepsilon \rho^b)
                    \left(\varepsilon \rho^b - \frac{1}{2}( \varepsilon \rho^b)^2
                    -  \varepsilon \tilde{\rho} - \frac{1}{2}( \varepsilon \tilde{\rho})^2 \right)  + \varepsilon \tilde{ \rho} - \varepsilon \rho^b \ dx + R_8\\
            &=   \int_{\mathbb{T}^3_x} \frac{1}{2} \left( \rho^b -\tilde{\rho}\right)^2 \ dx + R_8,
        \end{aligned}
    \end{equation}
    where
    \begin{equation}
        \begin{aligned}
            R_8 = \varepsilon \int_{\mathbb{T}^3_x} \left( C(\|\varepsilon \tilde{\rho}\|_{L^\infty(dx dt)}) (\rho^b)^3 +  C(\|\varepsilon \tilde{\rho}\|_{L^\infty(dx dt)}) (\tilde{\rho})^3\right) (1 + \varepsilon \rho^b) \ dx,
        \end{aligned}
    \end{equation}
    and $C(\|\varepsilon \tilde{\rho}\|_{L^\infty(dx dt)})$ is a constant depending only on $\|\varepsilon \tilde{\rho}\|_{L^\infty(dx dt)}$.

    And similarly,
    \begin{equation}\label{R9}
        \begin{aligned}
            \frac{1}{\varepsilon^2 }& \int_{\mathbb{T}^3_x} \rho_{f_{\varepsilon}} \left( -\frac{3}{2} \log\theta_{f_{\varepsilon}} + \frac{3}{2} \log(1+\varepsilon \tilde{\theta}) + \frac{3}{2(1+\varepsilon \tilde{\theta})} (\theta_{f_{\varepsilon}}-1-\varepsilon \tilde{\theta})  + \frac{|u_{f_{\varepsilon}}- \varepsilon \tilde{u}|^2}{2(1+\varepsilon \tilde{\theta})}\right)\ \ dx \\
            &=\frac{1}{\varepsilon^2 } \int_{\mathbb{T}^3_x} (1 + \varepsilon \rho^b) \frac{3}{2} \left( \varepsilon \tilde{\theta} - \frac{1}{2}(\varepsilon \tilde{\theta})^2 - \varepsilon \theta^b - \varepsilon^2 r_2 + \frac{1}{2} (\varepsilon \theta^b + \varepsilon^2 r_2)^2 + \frac{\varepsilon \theta^b + \varepsilon^2 r_2 - \varepsilon \tilde{\theta}}{1+ \varepsilon \tilde{ \theta}} \right)  \ dx \\
            &+  \varepsilon \int_{\mathbb{T}^3_x} \frac{3}{2} C(\|\varepsilon \tilde{\theta}\|_{L^{\infty}(dxdt)})\tilde{\theta}^3  - \frac{3}{2} C(\|\varepsilon \theta^b + \varepsilon^2 r_2\|_{L^{\infty}(dxdt)})  ( \theta^b + \varepsilon r_2)^3 \ dx \\
            &= \int_{\mathbb{T}^3_x}  \frac{3}{4} (\tilde{\theta} - \theta^b)^2 + R_9 ,
        \end{aligned}
    \end{equation}
    where
    \begin{equation}
        \begin{aligned}
            R_9& = \varepsilon \frac{3}{2}\int_{\mathbb{T}^3_x} - \frac{\tilde{\theta} r_2}{1+ \varepsilon \tilde{\theta}} + \theta^b r_2 + \frac{1}{2} \varepsilon r_2^2+ \frac{ (\theta^b -  \tilde{\theta})\tilde{\theta}}{1+ \varepsilon \tilde{\theta}}   \ dx \\
            &+ \varepsilon \frac{3}{2} \int_{\mathbb{T}^3_x} C(\|\varepsilon \tilde{\theta}\|_{L^{\infty}(dxdt)})\tilde{\theta}^3  -  C(\|\varepsilon \theta^b + \varepsilon^2 r_2\|_{L^{\infty}(dxdt)})  ( \theta^b + \varepsilon r_2)^3,
        \end{aligned}
    \end{equation}
    and $C(\|\varepsilon \tilde{\theta}\|_{L^{\infty}(dxdt)})$ along with $ C(\|\varepsilon \theta^b + \varepsilon^2 r_2\|_{L^{\infty}(dxdt)}) $ are constants depending only on $\|\varepsilon \tilde{\theta}\|_{L^{\infty}(dxdt)}$ and $\|\varepsilon \theta^b + \varepsilon^2 r_2\|_{L^{\infty}(dxdt)}$ respectively.

    Moreover, for the rest terms, we have
    \begin{equation}\label{R10}
        \begin{aligned}
            \frac{1}{\varepsilon^2 } \int_{\mathbb{T}^3_x} \rho_{f_{\varepsilon}}   \frac{|u_{f_{\varepsilon}}- \varepsilon \tilde{u}|^2}{2(1+\varepsilon \tilde{\theta})} \ dx
            = \frac{1}{2}\int_{\mathbb{T}^3_x} (\tilde{u}-u^b)^2 \ dx + R_{10},
        \end{aligned}
    \end{equation}
    where
    \begin{equation}
        \begin{aligned}
            R_{10}= \varepsilon \frac{1}{2}\int_{\mathbb{T}^3_x} \frac{(u^b - \tilde{u})(r_1 - \tilde{\theta}) + \varepsilon r_1}{1+ \varepsilon \tilde{\theta}}(1+ \varepsilon \rho^b) \ dx .
        \end{aligned}
    \end{equation}
    Combing \eqref{R8}, \eqref{R9} and \eqref{R10} completes the proof.
\end{proof}

\begin{proof}[\textbf{Proof of Lemma \ref{D_q_equivlence}}]
    Note that
    \begin{equation} \label{dissipation_temp1}
        \begin{aligned}
            \int^t_0 \int_{\mathbb{T}^3_x} \frac{1}{\varepsilon^4} D(f_{\varepsilon}) dxds =
            \frac{1}{4} \frac{1}{\varepsilon^4}\int^t_0 \int_{\mathbb{T}^3_x} \int_{\mathbb{R}^3_v}  \int_{\mathbb{R}^3_v}  \int_{\mathbb{S}^2}  \left(f^{\prime}_{\varepsilon 1} f^{\prime}_{\varepsilon }  - f_{\varepsilon 1} f_{\varepsilon} \right) \log \frac{f^{\prime}_{\varepsilon 1} f^{\prime}_{\varepsilon } }{f_{\varepsilon 1} f_{\varepsilon}} b\ d\sigma dv dv_1 dx ds, \\
            =\frac{1}{4} \frac{1}{\varepsilon^2} \int^t_0 \int_{\mathbb{T}^3_x} \int_{\mathbb{R}^3_v}  \int_{\mathbb{R}^3_v}  \int_{\mathbb{S}^2} q_{\varepsilon} \left( \log {f^{\prime}_{\varepsilon 1} f^{\prime}_{\varepsilon } }- \log{f_{\varepsilon 1} f_{\varepsilon}} \right)MM_1b\ d\sigma dv dv_1 dx ds,
        \end{aligned}
    \end{equation}
    and that
    \begin{equation}
        \begin{aligned}
            \log(f^{\prime}_{\varepsilon 1} f^{\prime}_{\varepsilon })
            &= \log(M^{\prime }_1 M^{\prime} + \varepsilon M^{\prime }_1 M^{\prime} g^{\prime}_{\varepsilon 1}+ \varepsilon M^{\prime }_1 M^{\prime} g^{\prime}_{\varepsilon}+ \varepsilon^2 M^{\prime }_1 M^{\prime}g^{\prime}_{\varepsilon 1}g^{\prime}_{\varepsilon })\\
            &=\log M^{\prime }_1 M^{\prime}  + \varepsilon g^{\prime}_{\varepsilon 1} + \varepsilon g^{\prime}_{\varepsilon} + \varepsilon^2 g^{\prime}_{\varepsilon } g^{\prime}_{\varepsilon 1} - \frac{1}{2} \left( \varepsilon g^{\prime}_{\varepsilon 1} + \varepsilon g^{\prime}_{\varepsilon} + \varepsilon^2 g^{\prime}_{\varepsilon } g^{\prime}_{\varepsilon 1}\right)^2 + r_3,\\
            \log(f_{\varepsilon 1} f_{\varepsilon })
            &= \log(M_1 M + \varepsilon M_1 M g_{\varepsilon 1}+ \varepsilon M_1 M g_{\varepsilon}+ \varepsilon^2 M_1 Mg_{\varepsilon 1}g_{\varepsilon })\\
            &=\log M_1 M  + \varepsilon g_{\varepsilon 1} + \varepsilon g_{\varepsilon} + \varepsilon^2 g_{\varepsilon } g_{\varepsilon 1} - \frac{1}{2} \left( \varepsilon g_{\varepsilon 1} + \varepsilon g_{\varepsilon} + \varepsilon^2 g_{\varepsilon } g_{\varepsilon 1}\right)^2 + r_4,\\
        \end{aligned}
    \end{equation}
    where
    \begin{equation}
        \begin{aligned}
            r_3 = \frac{\varepsilon^3}{1+ \tau\left( \varepsilon g^{\prime}_{\varepsilon 1} + \varepsilon g^{\prime}_{\varepsilon} + \varepsilon^2 g^{\prime}_{\varepsilon } g^{\prime}_{\varepsilon 1}\right)}  \left( g^{\prime}_{\varepsilon 1} + g^{\prime}_{\varepsilon} + \varepsilon g^{\prime}_{\varepsilon } g^{\prime}_{\varepsilon 1}\right)^3,\\
            r_4 = \frac{\varepsilon^3}{1+ \tau\left( \varepsilon g_{\varepsilon 1} + \varepsilon g_{\varepsilon} + \varepsilon^2 g_{\varepsilon } g_{\varepsilon 1}\right)}  \left(g_{\varepsilon 1} + g_{\varepsilon} + \varepsilon g_{\varepsilon } g_{\varepsilon 1}\right)^3,
        \end{aligned}
    \end{equation}
    for some $\tau \in [0,1]$.

    Therefore, we have
    \begin{equation}\label{dissipation_temp2}
        \begin{aligned}
            \frac{1}{\varepsilon^2} \left( \log {f^{\prime}_{\varepsilon 1} f^{\prime}_{\varepsilon } }- \log{f_{\varepsilon 1} f_{\varepsilon}} \right) &= q_{\varepsilon} -\frac{1}{2} q_{\varepsilon} \left(\varepsilon g^{\prime}_{\varepsilon 1} + \varepsilon g^{\prime}_{\varepsilon} + \varepsilon^2 g^{\prime}_{\varepsilon } g^{\prime}_{\varepsilon 1} + \varepsilon g^{\prime}_{\varepsilon 1} + \varepsilon g^{\prime}_{\varepsilon} + \varepsilon^2 g^{\prime}_{\varepsilon } g^{\prime}_{\varepsilon 1} \right)\\
             &+ \frac{1}{\varepsilon^2}r_3 - \frac{1}{\varepsilon^2}r_4.
        \end{aligned}
    \end{equation}

    Combing \eqref{dissipation_temp1} and \eqref{dissipation_temp2} yields that
    \begin{equation}
        \int^t_0 \int_{\mathbb{T}^3_x} \frac{1}{\varepsilon^4} D(f_{\varepsilon})\ dxds = \frac{1}{4}  \int^t_0 \int_{\mathbb{T}^3_x} \langle \langle q_{\varepsilon}^2  \rangle \rangle \ dx ds + R_{11}.
    \end{equation}
    
    \end{proof}

\begin{proof}[\textbf{Proof of Lemma \ref{q_asymptotics}}]
    Using \eqref{perturbed_Boltzmann}, we have
    \begin{equation}\label{appendix_temp1}
        \begin{aligned}
            \langle \varepsilon \partial_t g_{\varepsilon}, \hat{A} \rangle +  \langle v\mdot \nabla_x g_{\varepsilon}, \hat{A} \rangle  = \langle\!\langle \hat{A} q_{\varepsilon} \rangle\!\rangle .\\
        \end{aligned}
    \end{equation}
    Since, by \eqref{local_viscosity_and_heat_conductivity} and \eqref{v_dot_nabla_g} we have 
    \begin{equation}\label{appendix_temp2}
        \begin{aligned}
            \langle v\mdot \nabla_x g_{\varepsilon}, \hat{A} \rangle& = \langle v\mdot \nabla_x \Lpre g_{\varepsilon}, \hat{A} \rangle + \langle v\mdot \nabla_x \Loth g_{\varepsilon}, \hat{A} \rangle\\
            &= \mu \sigma(u^b) + \langle v\mdot \nabla_x \Loth g_{\varepsilon}, \hat{A} \rangle,
        \end{aligned}
    \end{equation}
    \eqref{appendix_temp1} and \eqref{appendix_temp2} then leads to 
    \begin{equation}
        \begin{aligned}
            \langle\!\langle \hat{A} q_{\varepsilon} \rangle\!\rangle - \mu \sigma(u^b) = \langle v \mdot \nabla_x \Loth g_{\varepsilon} , \hat{A}\rangle + \langle \varepsilon \partial_t g_{\varepsilon} , \hat{A}\rangle.
        \end{aligned}
    \end{equation}

    And similarly,
    \begin{equation}
        \begin{aligned}
            \langle\!\langle \hat{B} q_{\varepsilon} \rangle\!\rangle -\frac{5}{2}\kappa \nabla_x \theta^b = \langle v \mdot \nabla_x \Loth g_{\varepsilon} , \hat{B}\rangle + \langle \varepsilon \partial_t g_{\varepsilon} , \hat{B}\rangle.
        \end{aligned}
    \end{equation}
\end{proof}
\bibliography{ref.bib}
        \bibliographystyle{amsplain}
\end{document}